\documentclass[11pt]{amsart}

\usepackage[T1]{fontenc}
\usepackage[utf8]{inputenc}
\usepackage{amsmath,amssymb,mathtools}
\usepackage{enumitem}
\usepackage{graphicx}
\usepackage{microtype}
\usepackage[hidelinks]{hyperref}

\newtheorem{theorem}{Theorem}[section]
\newtheorem{proposition}[theorem]{Proposition}
\newtheorem{lemma}[theorem]{Lemma}
\newtheorem{corollary}[theorem]{Corollary}
\theoremstyle{definition}
\newtheorem{definition}[theorem]{Definition}
\newtheorem{example}[theorem]{Example}
\theoremstyle{remark}
\newtheorem{remark}[theorem]{Remark}

\newcommand{\R}{\mathbb R}

\newcommand{\intt}{\operatorname{int}}
\newcommand{\bd}{\operatorname{bd}}

\newcommand{\cone}{\operatorname{cone}}

\newcommand{\cF}{\mathcal F}

\begin{document}

\title[Directionally Fine Families of Cones]{Directionally Fine Families of Cones:\\
Construction Principles, Approximation, and Separation}

\author{Fernando García-Castaño}
\address{Department of Mathematics, University of Alicante, Alicante, Spain}
\email{fernando.gc@ua.es}
\thanks{ORCID: 0000-0002-8352-8235}

\author{Miguel Ángel Melguizo-Padial}
\address{Department of Mathematics, University of Alicante, Alicante, Spain}
\email{ma.mp@ua.es}
\thanks{ORCID: 0000-0003-0303-791X}

\date{September 2026}

\begin{abstract}
We introduce the notion of a directionally fine family of cones, a local
geometric property expressing the possibility of selecting arbitrarily
narrow cones around every direction of a normed space. We show that this
local property provides a common mechanism for finite separation and
approximation. In particular, directionally compact sets can be separated
from closed cones by finitely many members of the family, while locally
compact cones admit finite outer approximations with quantitative
Hausdorff control on their sets of directions and on bounded sections.
These finite constructions lead naturally to max-type positively
homogeneous separators and do not require convexity of the cones to be
separated.

We then study two complementary realizations of the framework. For the
classical family of Bishop--Phelps cones, directional fineness is
characterized by the requirement that every point of the unit sphere be
a denting point of the unit ball, equivalently, by property~$(G)$; in
Banach spaces this amounts to rotundity together with the Kadec property.
On the other hand, we introduce transversal coercivity as a general axial
construction principle. It yields directionally fine families in arbitrary
normed spaces, without geometric assumptions on the norm, and includes
uniform axial deviation and norm-normalized axial cones as canonical
models. Finally, the explicit structure of the
former yields penalization and recovery results for optimization over locally
compact cones, together with quantitative approximation and convergence estimates.
\end{abstract}


\maketitle

\section{Introduction}

Separation by cones and positively homogeneous functionals is a classical
tool in optimization, vector optimization, and variational analysis. It
plays a central role in scalarization, proper efficiency, and in the study
of geometric relations between feasible, ordering, and tangent cones.
Bishop--Phelps cones constitute one of the standard models in this context
and have been extensively used in nonlinear separation and vector
optimization; see, among others,
\cite{Phelps1974,Jahn2009,Kasimbeyli2010,HaJahn2017}.
More recent separation results continue to exploit the geometry of such
cones and the properties of their norm-bases; see, for instance,
\cite{GarciaCastanoEtAl2025}.

Most classical cone-separation schemes are formulated in terms of a
particular class of separating cones, and frequently seek a single cone
with the required separation properties. The viewpoint adopted here is
different. We start from a local question on the space of directions:
given a unit direction and an arbitrarily small directional scale, does a
prescribed family contain a cone having that direction in its interior and
remaining inside the prescribed scale? This leads to the notion of a
\emph{directionally fine family}. Geometrically, directional fineness is a
local-base property on the unit sphere, or equivalently on the space of
positive rays.

The main purpose of the paper is to show that this local property provides
a general mechanism for passing from local directional constructions to
finite separation and approximation results. The passage from local to
finite is governed by compactness of the relevant set of directions. Thus
radial size plays no essential role: what matters is the geometry of the
normalized image of the sets on the unit sphere.

We first develop the functional framework associated with directionally
fine families. Every closed cone selected around a prescribed direction
admits a continuous positively homogeneous defining function which can be
normalized at that direction and controlled by the norm. Consequently, the
functional structure required for separation does not have to be imposed
as an additional hypothesis. In the convex case, the defining function
may instead be chosen superlinear, preserving the normalization and the
required one-sided norm estimate. This connects the geometric selection
of cones with nonlinear scalarization by positively homogeneous
representing functions; compare, for instance,
\cite{GuntherKhazayelStrugariuTammer2026}.

A first consequence of directional fineness is a local-to-global
separation principle. A point outside a closed cone can be surrounded by
an arbitrarily narrow member of the family which intersects the closed
cone only at the origin. When the normalized directions of a set are
compact, finitely many such local cones suffice. Their defining functions
can then be combined through a maximum, producing a single continuous
positively homogeneous separator. This leads naturally to the notion of
directional compactness. We also examine its relation with weak compactness
and show that, in Kadec spaces, where the weak and norm topologies coincide on the unit sphere, directional weak compactness and
directional compactness coincide.

The same local-to-finite mechanism also yields quantitative outer approximation
of locally compact cones. If $L$ is locally compact, its norm-base is compact and can therefore be
covered by finitely many local models from any directionally fine family.
For every prescribed conic scale $\varepsilon$, this produces a finite
outer approximation
\[
L\subset D_L^\varepsilon\subset L_\varepsilon,
\]
together with quantitative estimates
\[
d_H(D_L^\varepsilon\cap S_X,B_L)\leq2\varepsilon,
\]
and, on bounded sections,
\[
d_H(D_L^\varepsilon\cap RB_X,L\cap RB_X)
\leq2R\varepsilon.
\]
The same construction gives finite separation of disjoint cones and,
when both cones are locally compact, symmetric separation with uniform
homogeneous margins.

The finite character of these results is essential. In general, a
nonconvex cone may have several separated directional branches whose
simultaneous containment in a single convex cone necessarily introduces
directions that one wishes to exclude. Thus passing from one separating
cone to a finite family is not merely a technical consequence of a
compactness argument. Finite unions genuinely enlarge the geometries that
can be treated, and the resulting max-type separator reflects this
disjunctive structure. In this respect, the present approach differs from
separation schemes based on a single Bishop--Phelps cone. The two
viewpoints are complementary: the latter seek a suitable individual
convex separator under appropriate geometric conditions, whereas here
local separating cones are allowed to be assembled into a finite
nonconvex model.

A natural question is then which familiar cone families are directionally
fine. We give a complete answer for the classical Bishop--Phelps family.
We prove that this family is directionally fine if and only if every point
of the unit sphere is a denting point of the unit ball, that is, if and
only if the norm has the classical property~$(G)$ of Fan and Glicksberg
\cite{FanGlicksberg1958}. In Banach spaces, this is equivalent to
rotundity together with the Kadec property
\cite{LinLinTroyanski1986}. In particular, the Bishop--Phelps family is directionally
fine in every locally uniformly rotund space, and hence in every
uniformly convex space. Thus the result applies, in particular, to Hilbert
spaces and to the classical spaces $\ell^p$ and $L^p(\mu)$,
$1<p<\infty$, with their usual norms. In finite-dimensional spaces, the
condition reduces to strict convexity. Once this characterization is established, all the abstract finite approximation and separation results immediately provide corresponding
Bishop--Phelps constructions.

The Bishop--Phelps characterization also shows that directional fineness
of a prescribed classical family may depend substantially on the geometry
of the norm. We therefore consider a complementary problem: whether
directionally fine families can be constructed in arbitrary normed spaces,
without rotundity, dentability, smoothness, or related assumptions on the
norm. To this end, we develop an axial construction principle for arbitrary normed spaces, based on a distinguished direction, a norming functional, and a positively
homogeneous measure of deviation from that direction. The essential
requirement is a coercivity condition controlling transversal departures
from the axis. We refer to this condition as \emph{transversal
coercivity}.

This axial principle is used in both directions. On the one hand, every
closed convex cone sufficiently localized around a prescribed direction
admits an exact representation through a transversally coercive deviation.
On the other hand, every such deviation generates localized cones by two
natural normalization procedures. These constructions yield, in
particular, two canonical directionally fine families in arbitrary normed
spaces: the uniform axial deviation family and the norm-normalized axial
family. The former always yields convex cones, whereas convexity is not guaranteed
for the latter. More general deviations lead to further examples, so
transversal coercivity should be regarded as a general construction
principle rather than as the definition of a single cone model.

Accordingly, the abstract theory developed here may be summarized by the
scheme
\[
\begin{aligned}
\text{local cone construction}
&\Longrightarrow \text{directionally fine family}\\
&\Longrightarrow \text{finite approximation and separation}.
\end{aligned}
\]
When the local cones admit explicit positively homogeneous defining
functions, finite unions are represented by the maximum of these
functions. Thus every concrete directionally fine family yields its own
explicit realization of the same separation mechanism.

As a simple optimization consequence, we finally specialize the finite
outer approximation to uniform axial deviation cones. Their explicit
structure provides an error bound for the finite approximating cone.
Combined with Lipschitz continuity of the objective, this yields an exact
penalization of the extended problem and a recovery procedure producing a
feasible point for the original conic problem. The recovered point is
quantitatively close to the penalized solution, with explicit estimates for
both the optimality gap and the approximation of the optimal value.
Under uniform boundedness of penalized minimizers, these estimates further
yield convergence of optimal values and subsequential convergence to
solutions of the original problem. This application is included mainly to
illustrate how the geometric construction can be combined with standard
optimization tools once an explicit directionally fine family has been
chosen.

The paper is organized as follows. Section~2 introduces defining functions
and directionally fine families, including the superlinear representation
available in the convex case. Section~3 develops the finite separation and
approximation theory, including directional compactness and Hausdorff
approximation of locally compact cones. Section~4 studies the
Bishop--Phelps family and characterizes its directional fineness in terms
of denting points and property~(G). Finally, Section~5 develops the axial
representation and construction principles based on transversal
coercivity, derives the canonical uniform axial deviation and
norm-normalized families, and applies the former to penalization,
approximate recovery, and convergence in conically constrained
optimization.

\section{Defining Functions and Directionally Fine Families}
This section introduces the basic notation and the functional-geometric framework that underlies the separation and approximation results developed later.

Throughout the paper, $X$ is a real normed space, $X^*$ its continuous
dual, and $0_X$ its origin. We write $B_X:=\{x\in X:\|x\|\leq1\}$, $S_X:=\{x\in X:\|x\|=1\}$, $S_{X^*}:=\{f\in X^*:\|f\|=1\}$, and $B(x,r):=x+rB_X$, $x\in X$, $r>0$. We denote by
$\operatorname{int}A$ and $\operatorname{bd}A$ the interior and boundary
of a set $A$, respectively, by $w=\sigma(X,X^*)$ the weak topology on
$X$, and set $\mathbb R_+:=[0,+\infty)$ and $t_+:=\max\{t,0\} \quad \text{for} \quad t\in\R$.

For nonempty sets $A,B\subset X$ and $x\in X$, set $d(x,A):=\inf_{a\in A}\|x-a\|$, $d(A,B):=\inf_{a\in A}d(a,B)$. For nonempty closed bounded sets $A,B\subset X$, their Hausdorff distance
is
\[
d_H(A,B):=
\max\left\{\sup_{a\in A}d(a,B),\sup_{b\in B}d(b,A)\right\}.
\]

A nonempty set $C\subset X$ is called a cone if $tC\subset C$ for every
$t\geq0$. A cone is proper if $\{0_X\}\subsetneq C\subsetneq X$, pointed
if $C\cap(-C)=\{0_X\}$, and solid if $\operatorname{int}C\neq\varnothing$.
Unless otherwise stated, all cones considered throughout the paper are
assumed to be proper. For $A\subset X$, set
$\operatorname{cone}(A):=\{ta:t\geq0,\ a\in A\}$.

A mapping $\varphi:X\to\R$ is positively homogeneous if $\varphi(tx)=t\varphi(x)$, $\forall t\ge0$, $x\in X$. A positively homogeneous mapping is superlinear if $\varphi(x+y)\ge\varphi(x)+\varphi(y)$, $\forall x,y\in X$. A mapping $p:X\to\mathbb R$ is subadditive if $p(x+y)\leq p(x)+p(y)$ for all $x,y\in X$. A mapping $f:X\to\R$ is Lipschitz continuous on $X$ if there exists $L\ge 0$ such that
\(
|f(x)-f(y)|
\le
L\|x-y\|
\quad
\forall x,y\in X.
\) In this case, $f$ is also called $L$-Lipschitz.

\begin{definition}
\label{def:defining-function}
Let $C\subset X$ be a cone. A mapping $\varphi_C:X\to\R$ is called a
\emph{defining function} for $C$ if
$C=\{x\in X:\varphi_C(x)\geq0\}$.
\end{definition}

\begin{remark}
\label{rem:closedness_essential}
From now on, cones will be assumed to be closed whenever a continuous
defining function is involved. Indeed, if $\varphi:X\to\R$ is continuous,
then $\{\varphi\geq0\}$ is closed. Thus closedness is the natural
assumption for the functional representation considered below.
\end{remark}

The next result is the key simplification of the abstract framework.

\begin{lemma}
\label{lem:adapted_defining_function}
Let $C\subset X$ be a closed cone and let $u\in S_X\cap\intt C$. Then
there exists a continuous positively homogeneous mapping
$\varphi_{C,u}:X\to\R$ such that
\[
C=\{\varphi_{C,u}\geq0\},\qquad
\intt C=\{\varphi_{C,u}>0\},\qquad
\bd C=\{\varphi_{C,u}=0\},
\]
with $\varphi_{C,u}(u)=1$, and $|\varphi_{C,u}(x)|\leq\|x\|$ for every $x\in X$.
\end{lemma}

\begin{proof}
Set $D:=X\setminus C$ and
$\psi_C(x):=d(x,D)-d(x,C)$. Since $C$ is a proper cone,
$D\neq\varnothing$ and $tC=C$, $tD=D$ for every $t>0$. Hence
$d(tx,C)=t\,d(x,C)$ and $d(tx,D)=t\,d(x,D)$, so $\psi_C$ is positively
homogeneous; it is continuous by the Lipschitz continuity of distance
functions, and $\psi_C(0_X)=0$ since $0_X\in\overline D\cap C$.
Moreover, $\psi_C>0$ on $\intt C$, $\psi_C=0$ on $\bd C$, and
$\psi_C<0$ on $X\setminus C$.

Since $u\in\intt C$, put $a:=\psi_C(u)>0$ and
$\widehat\psi_C:=\psi_C/a$, and define
\[
\varphi_{C,u}(x):=
\max\{-\|x\|,\min\{\|x\|,\widehat\psi_C(x)\}\}.
\]
Then $\varphi_{C,u}$ is continuous and positively homogeneous,
$|\varphi_{C,u}(x)|\leq\|x\|$, and the truncation preserves the positive,
zero, and nonnegative level sets of $\psi_C$. Finally,
$\widehat\psi_C(u)=\|u\|=1$, so $\varphi_{C,u}(u)=1$.
\end{proof}

\begin{remark}
\label{rem:oriented-distance}
The function
\[
\psi_C(x)=d(x,X\setminus C)-d(x,C)
\]
is, up to sign, the classical oriented distance function; see, e.g.,
\cite{DelfourZolesio1994,DelfourZolesio2004}. In the conic setting,
Seeger~\cite{Seeger2025} uses oriented distances to study erosions and
dilations of proper convex cones in finite-dimensional Euclidean spaces.
Here the role is different: for an arbitrary closed cone in a normed
space, without convexity assumptions, $\psi_C$ provides a continuous
positively homogeneous defining function whose normalization and
truncation yield the adapted defining function used above.
\end{remark}

\begin{definition}
\label{def:adapted_defining_function}
Let $C$ be a closed cone and $u\in S_X\cap\intt C$. A continuous
positively homogeneous defining function $\varphi_{C,u}$ satisfying
$\varphi_{C,u}(u)=1$ and
$\varphi_{C,u}(x)\leq\|x\|$ for every $x\in X$ is called an
\emph{adapted defining function for $C$ at the direction $u$}.
\end{definition}

\begin{definition}
\label{def:directionally_fine}
Let $\cF$ be a family of closed cones in $X$. We say that $\cF$ is
\emph{directionally fine} if, for every $u\in S_X$ and
$\varepsilon\in(0,1)$, there exists $C\in\cF$ such that
\[
u\in\intt C\subset C\subset\cone B(u,\varepsilon).
\]
If $C$ can always be chosen convex, we say that $\cF$ is
\emph{convex directionally fine}.
\end{definition}

\begin{remark}
The formulation on $S_X$ is equivalent to the corresponding directional formulation around an arbitrary nonzero point. Indeed, if $x\ne0_X$ and $u=x/\|x\|$, then
\[
\cone B(x,\delta)
=
\cone B\!\left(u,\frac{\delta}{\|x\|}\right).
\]
The restriction $\varepsilon<1$ is useful because $0_X\notin B(u,\varepsilon)$ and $\cone B(u,\varepsilon)$ is a proper cone.
\end{remark}

By Lemma~\ref{lem:adapted_defining_function}, every cone selected by
directional fineness automatically admits an adapted defining function
at the prescribed direction, with the normalization and norm control
required below.

\begin{remark}
\label{rem:defining-vs-representing}
The defining functions used here are closely related, up to a sign
convention, to the representing functionals of
\cite[Definition~4]{GuntherKhazayelStrugariuTammer2026}.
Their role is different, however: rather than being introduced primarily
for scalarization, our adapted defining functions are attached to
locally selected separating cones and are normalized and norm-controlled
so as to encode their directional localization.
\end{remark}

We next isolate the convex case. Here the defining function can be chosen superlinear without losing the normalization or the sharp norm control.

\begin{proposition}
\label{prop:superlinear_adapted}
Let $C\subset X$ be a closed convex cone and let
$u\in S_X\cap\intt C$. Define
\[
C^+:=\{x^*\in X^*:x^*(c)\geq0\ \forall c\in C\},\qquad
B_u:=\{x^*\in C^+:x^*(u)=1\},
\]
and
\[
\varphi^{\rm sl}_{C,u}(x):=\inf_{x^*\in B_u}x^*(x).
\]
Then $B_u$ is nonempty and norm bounded, and
$\varphi^{\rm sl}_{C,u}$ is finite-valued, continuous, positively
homogeneous, and superlinear. Moreover,
\[
C=\{\varphi^{\rm sl}_{C,u}\geq0\},\qquad
\intt C=\{\varphi^{\rm sl}_{C,u}>0\},\qquad
\bd C=\{\varphi^{\rm sl}_{C,u}=0\},
\]
and $\varphi^{\rm sl}_{C,u}(u)=1$. If $C\subset\cone B(u,\varepsilon)$ for some $\varepsilon\in(0,1)$,
then $\varphi^{\rm sl}_{C,u}(x)\leq\|x\|$ for every $x\in X$; hence
$\varphi^{\rm sl}_{C,u}$ is adapted.
\end{proposition}

\begin{proof}
Choose $\rho>0$ with $u+\rho B_X\subset C$. Since $C$ is a proper
closed convex cone, separation yields a nonzero $x^*\in C^+$. If
$x^*(u)=0$, then $0\leq x^*(u\pm\rho h)=\pm\rho x^*(h)$ for every
$h\in B_X$, forcing $x^*=0$. Thus $x^*(u)>0$, and normalization gives
$B_u\neq\varnothing$. Moreover, for $x^*\in B_u$ and $h\in B_X$,
$0\leq1\pm\rho x^*(h)$, so
$\|x^*\|\leq1/\rho$. Hence
$M:=\sup_{x^*\in B_u}\|x^*\|<+\infty$ and
$|\varphi^{\rm sl}_{C,u}(x)-\varphi^{\rm sl}_{C,u}(y)|
\leq M\|x-y\|$; thus $\varphi^{\rm sl}_{C,u}$ is finite-valued and
continuous. Positive homogeneity is immediate, and
$\varphi^{\rm sl}_{C,u}(x+y)\geq
\varphi^{\rm sl}_{C,u}(x)+\varphi^{\rm sl}_{C,u}(y)$, so it is
superlinear. Also $\varphi^{\rm sl}_{C,u}(u)=1$.

Clearly $\varphi^{\rm sl}_{C,u}\geq0$ on $C$. If $x\notin C$,
separation gives $y^*\in C^+$ with $y^*(x)<0$; as above $y^*(u)>0$,
and therefore
$\varphi^{\rm sl}_{C,u}(x)\leq y^*(x)/y^*(u)<0$. Hence
$C=\{\varphi^{\rm sl}_{C,u}\geq0\}$.

If $x\in\intt C$, choose $\delta>0$ with $x+\delta B_X\subset C$.
For $x^*\in B_u$ and $h\in B_X$,
$0\leq x^*(x-\delta h)$, whence
$x^*(x)\geq\delta\|x^*\|\geq\delta$. Thus
$\varphi^{\rm sl}_{C,u}(x)>0$. The converse follows from continuity, so
$\intt C=\{\varphi^{\rm sl}_{C,u}>0\}$ and, since $C$ is closed,
$\bd C=\{\varphi^{\rm sl}_{C,u}=0\}$.

Finally, assume $C\subset\cone B(u,\varepsilon)$ and choose
$f\in S_{X^*}$ with $f(u)=1$. If $0_X\neq c=ty\in C$, with $t>0$ and
$y\in B(u,\varepsilon)$, then
$f(y)\geq1-\|y-u\|\geq1-\varepsilon>0$. Hence $f\in B_u$, and therefore
$\varphi^{\rm sl}_{C,u}(x)\leq f(x)\leq\|x\|$ for every $x\in X$.
\end{proof}

\begin{remark}
\label{rem:one-sided-norm-control}
The estimate in Proposition~\ref{prop:superlinear_adapted} is
one-sided. Indeed, norm boundedness of $B_u$ only gives
\[
-M\|x\|\leq\varphi^{\rm sl}_{C,u}(x)\leq\|x\|,
\qquad
M:=\sup_{x^*\in B_u}\|x^*\|<+\infty,
\]
and in general $M$ need not equal $1$. Hence the stronger estimate
$|\varphi^{\rm sl}_{C,u}|\leq\|\cdot\|$ does not follow, nor is it
needed for adaptedness or for the separation results below.
\end{remark}

\begin{remark}\label{rem:oriented-distance-dual}
For convex cones, Proposition~\ref{prop:superlinear_adapted}
is also related to the oriented-distance representation discussed in
Remark~\ref{rem:oriented-distance}. In the finite-dimensional Euclidean setting, Seeger~\cite{Seeger2025}
represents the oriented distance to a proper convex cone by an extremum
over norm-normalized elements of its polar cone. With the sign convention
of Remark~\ref{rem:oriented-distance}, this yields an infimum representation
over norm-normalized elements of the positive dual cone. Thus both
constructions have a similar dual structure, but with different
normalizations: the oriented-distance representation uses norm-normalized
dual elements, whereas
\[
\varphi^{\rm sl}_{C,u}(x)
=
\inf\{x^*(x):x^*\in C^+,\ x^*(u)=1\},
\]
uses the distinguished direction $u$. This yields
$\varphi^{\rm sl}_{C,u}(u)=1$ and extends naturally to arbitrary normed
spaces.
\end{remark}

\begin{corollary}
\label{cor:convex_directionally_fine}
If $\cF$ is convex directionally fine, then for every $u\in S_X$ and every $\varepsilon\in(0,1)$ one can choose $C\in\cF$ and a continuous superlinear adapted defining function $\varphi_{C,u}$ such that $u\in\intt C
\subset C
\subset\cone B(u,\varepsilon)$,
\[
C=\{\varphi_{C,u}\ge0\},
\qquad
\varphi_{C,u}(u)=1,
\qquad
\varphi_{C,u}(x)\le\|x\|,
\quad\forall x\in X.
\]
\end{corollary}

\begin{proof}
Choose a convex cone $C$ as in Definition~\ref{def:directionally_fine} and apply Proposition~\ref{prop:superlinear_adapted}.
\end{proof}

\section{Separation and Approximation by Directionally Fine Families}\label{sec:separation-approximation}
The results of this section develop the two main consequences of
directional fineness. We first pass from local directional separation to
finite separation of sets, and then use compactness of the set of
directions to obtain finite approximation and separation results for
cones.
\subsection{Separation from Closed Cones}~\label{subsec:finite-separation}
This subsection develops a local-to-global separation principle.
Directional fineness first provides separation around each individual
direction; compactness of the normalized directions then reduces the
construction to a finite family, whose defining functions can be combined
into a single max-type separator. The required continuous positively
homogeneous adapted defining functions are supplied by
Lemma~\ref{lem:adapted_defining_function}.

\begin{proposition}
\label{prop:closed_cone_point}
Let $\cF$ be directionally fine. Let $K\subset X$ be a closed cone and
let \(
x_0\in X\setminus K.
\)
Then, for every
\(
0<\varepsilon<d(x_0,K),
\)
there exist a cone $C\in\cF$ and an adapted defining function $\varphi$
for $C$ such that
\[
x_0\in\intt C
\subset
C
\subset
\cone B(x_0,\varepsilon),
\qquad
K\cap C=\{0_X\},
\]
and
\(
\varphi(x_0)=\|x_0\|.
\)
Moreover, for every
\(
0<\rho<\|x_0\|,
\)
the set
\(
U_\rho
:=
\{x\in X:\varphi(x)>\rho\}
\)
is an open neighborhood of $x_0$ satisfying
\(
U_\rho\subset\intt C
\)
and
\[
\varphi(k)\le0<\rho<\varphi(x),
\qquad
\forall k\in K,\quad
\forall x\in U_\rho.
\]
\end{proposition}

\begin{proof}
Set $u_0:=x_0/\|x_0\|$ and
$\eta:=\varepsilon/\|x_0\|$. Since $K$ is a cone,
$d(x_0,K)=\|x_0\|d(u_0,K)$. Moreover, since $0_X\in K$,
$d(x_0,K)\leq\|x_0\|$. Hence
$0<\eta<\min\{1,d(u_0,K)\}$. By directional fineness, there exists $C\in\cF$ such that $u_0\in\intt C\subset C\subset\cone B(u_0,\eta)$. Since $B(x_0,\varepsilon)=\|x_0\|B(u_0,\eta)$ and positive scaling
does not change the generated cone,
$\cone B(u_0,\eta)=\cone B(x_0,\varepsilon)$. Thus
$x_0\in\intt C\subset C\subset\cone B(x_0,\varepsilon)$.

By Lemma~\ref{lem:adapted_defining_function}, let $\varphi$ be an
adapted defining function for $C$ at $u_0$. Positive homogeneity and
normalization yield $\varphi(x_0)=\|x_0\|$. Since $\varepsilon<d(x_0,K)$, one has
$B(x_0,\varepsilon)\cap K=\varnothing$, and hence
$K\cap\cone B(x_0,\varepsilon)=\{0_X\}$; otherwise,
$0_X\neq k=ty\in K$, with $t>0$ and
$y\in B(x_0,\varepsilon)$, would imply $y=t^{-1}k\in K$.
Therefore $K\cap C=\{0_X\}$. Consequently,
$\varphi(k)<0$ for $k\in K\setminus\{0_X\}$, while
$\varphi(0_X)=0$, so $\varphi\leq0$ on $K$.

Finally, if $0<\rho<\|x_0\|$, then
$U_\rho:=\{\varphi>\rho\}$ is, by continuity, an open neighborhood
of $x_0$. Since $\rho>0$,
$U_\rho\subset\{\varphi>0\}=\intt C$, and therefore $\varphi(k)\leq0<\rho<\varphi(x)$,
$\forall k\in K,\ \forall x\in U_\rho$.
\end{proof}
\begin{remark}\label{rem:applications_closed_cone_point}\label{rem:nonconvex-cones}
Proposition~\ref{prop:closed_cone_point} provides nonlinear certificates of
exclusion for directions lying outside a closed, possibly nonconvex, cone,
without passing to its closed convex hull. This feature may be useful in
variational analysis when the relevant tangent or graphical cone is
nonconvex. In particular, it is compatible with Bouligand-type first-order
analysis and with graphical differentiation of set-valued mappings; see,
e.g., \cite{RockafellarWets1998,Sama2010}. Related nonconvex conic
structures arise naturally in disjunctive and complementarity systems; see,
e.g., \cite{OutrataKocvaraZowe1998}. In such settings, convexification may
introduce directions that are absent from the original geometry, whereas
Proposition~\ref{prop:closed_cone_point} works directly with the underlying
nonconvex cone.
\end{remark}

To pass from the preceding pointwise construction to a finite one, only
compactness of the normalized directions is needed. This motivates the
following notion.

\begin{definition}
\label{def:directionally_compact}
For $A\subset X\setminus\{0_X\}$, set $\Pi_{S_X}(A)
:=
\left\{\frac{a}{\|a\|}:a\in A\right\}$. We say that $A$ is \emph{directionally compact} (resp. \emph{directionally weakly compact}) if $\Pi_{S_X}(A)$ is compact (resp. weakly compact).
\end{definition}

The terminology ``directionally compact'' has previously appeared in
different, mainly local and variational, senses; see, for instance,
\cite{Penot1983,Sama2010}. Here, the condition imposes
no radial compactness on $A$; in particular, directionally compact sets may
be unbounded. The following example illustrates that directional compactness is compatible
with unboundedness and genuinely infinite-dimensional directional behavior.

\begin{example}
Let $X=\ell_2$, let $(e_n)$ be its canonical basis, and set
$u_n=(e_1+n^{-1}e_n)/\sqrt{1+n^{-2}}$ for $n\geq2$. Since
$u_n\to e_1$ in norm, the set
$D:=\{e_1\}\cup\{u_n:n\geq2\}\subset S_{\ell_2}$ is compact. Hence
$A:=\{tu:t>0,\ u\in D\}$ is directionally compact. Nevertheless, $A$
is unbounded and its directions span an infinite-dimensional subspace
of $\ell_2$.
\end{example}
We can now combine the pointwise separation furnished by
Proposition~\ref{prop:closed_cone_point} with compactness of the set of
directions.
\begin{theorem}
\label{thm:finite-separation-directionally-compact}
Let $\cF$ be directionally fine, let $K\subset X$ be a closed cone,
and let $A\subset X\setminus\{0_X\}$ be directionally compact with
$A\cap K=\varnothing$. Then there exist $C_i\in\mathcal F$ and adapted defining functions $\varphi_i$, $i=1,\ldots,n$, such that
$\varphi_i\leq0$ on $K$ and $K\cap C_i=\{0_X\}$ for every $i$.
Writing
$V_i:=\{x\in X:\varphi_i(x)>\frac12\|x\|\}$, one has
\[
\Pi_{S_X}(A)\subset\bigcup_{i=1}^n\{\varphi_i>1/2\},
\qquad
A\subset\bigcup_{i=1}^n V_i
\subset\bigcup_{i=1}^n\intt C_i.
\]
\end{theorem}

\begin{proof}
Set $D:=\Pi_{S_X}(A)$. Since $K$ is a cone and $A\cap K=\varnothing$,
one has $D\cap K=\varnothing$.
Since $D$ is compact and $K$ is closed, it follows that $d(D,K)>0$. Fix $\eta:=\frac12 d(D,K)>0$. Then $0<\eta<d(u,K)$, $\forall u\in D$. For each $u\in D$, apply
Proposition~\ref{prop:closed_cone_point} with $\varepsilon=\eta$ and
$\rho=1/2$. This gives
$C_u\in\cF$ and an adapted defining function $\varphi_u$ such that
$\varphi_u\leq0$ on $K$, $K\cap C_u=\{0_X\}$, and
$\{\varphi_u>1/2\}$ is a neighborhood of $u$ contained in $\intt C_u$.
Compactness of $D$ yields a finite subcover, which gives the first
inclusion after relabeling.

If $a\in A$, choose $i$ such that
$\varphi_i(a/\|a\|)>1/2$. Positive homogeneity gives
$\varphi_i(a)>\frac12\|a\|$, so $a\in V_i$. Finally,
$V_i\subset\{\varphi_i>0\}\subset\intt C_i$.
\end{proof}

The finite separating family obtained above can be encoded by a single
max-type positively homogeneous function.
\begin{corollary}
\label{cor:single-separator-directionally-compact}
Under the hypotheses of
Theorem~\ref{thm:finite-separation-directionally-compact}, let   $C_i\in\mathcal F$ and $\varphi_i$, $i=1,\ldots,n$,  be as there and set
$\Phi:=2\max_{1\leq i\leq n}\varphi_i$. Then $\Phi:X\to\R$ is
continuous and positively homogeneous, and
\[
\{\Phi\geq0\}=\bigcup_{i=1}^n C_i,\qquad
\{\Phi\geq0\}\cap K=\{0_X\},\qquad
\Phi(k)<0<\|a\|<\Phi(a),
\]
for every $k\in K\setminus\{0_X\}$ and $a\in A$.
\end{corollary}

\begin{proof}
Continuity and positive homogeneity are immediate. Since
$C_i=\{\varphi_i\geq0\}$, one has
$\{\Phi\geq0\}=\bigcup_{i=1}^n C_i$, whose intersection with $K$ is
$\{0_X\}$. If $k\in K\setminus\{0_X\}$, then $k\notin C_i$ and hence
$\varphi_i(k)<0$ for every $i$, so $\Phi(k)<0$. Finally, for each
$a\in A$, the preceding theorem gives an index $i$ such that
$\varphi_i(a)>\frac12\|a\|$, and therefore $\Phi(a)>\|a\|$.
\end{proof}
For compact sets, the positive distance from the origin converts the
homogeneous margin in the preceding corollary into a fixed positive
separation level.
\begin{corollary}
\label{cor:single-separator-compact}
Let $\cF$ be directionally fine, let $K\subset X$ be a closed cone,
and let $A\subset X$ be compact with $A\cap K=\varnothing$. Then there
exist $C_1,\ldots,C_n\in\cF$ and a continuous positively homogeneous
function $\Phi:X\to\R$ such that
$A\subset\bigcup_{i=1}^n\intt C_i$, $K\cap C_i=\{0_X\}$ for every $i$,
and $\{\Phi\geq0\}=\bigcup_{i=1}^n C_i$. Moreover,
\[
\Phi(k)\leq0<1<\Phi(a),
\quad \forall k\in K,\ a\in A,\qquad
\Phi(k)<0,
\quad \forall k\in K\setminus\{0_X\}.
\]
\end{corollary}

\begin{proof}
Since $0_X\in K$ and $A\cap K=\varnothing$, one has $0_X\notin A$.
The normalization map is continuous on $A$, so $A$ is directionally
compact, and
$\delta:=d(0_X,A)=\min_{a\in A}\|a\|>0$.

Apply Theorem~\ref{thm:finite-separation-directionally-compact} and
Corollary~\ref{cor:single-separator-directionally-compact} to obtain
$C_1,\ldots,C_n$ and a separator $\Psi$, and set
$\Phi:=\delta^{-1}\Psi$. This positive rescaling preserves all signs
and level sets. Moreover,
$\Phi(a)>\|a\|/\delta\geq1$ for every $a\in A$, which proves the
result.
\end{proof}

The preceding separation result is formulated in terms of norm
compactness of the set of directions. It is therefore natural to ask
to what extent this hypothesis can be expressed in terms of weak
compactness. The next observation isolates precisely the additional
topological condition that is needed.
\begin{proposition}
\label{prop:directional-compactness-weak}
Let $A\subset X\setminus\{0_X\}$ and set
$D:=\Pi_{S_X}(A)$. Then the following assertions are
equivalent:
\begin{enumerate}
\item[(i)] $A$ is directionally compact;
\item[(ii)] $D$ is weakly compact and the identity mapping $\operatorname{id}:
(D,w) \longrightarrow (D,\|\cdot\|)$ is continuous.
\end{enumerate}
\end{proposition}

\begin{proof}
If $D$ is norm-compact, then it is weakly
compact, and $\operatorname{id}:(D,\|\cdot\|)\to(D,w)$, is a continuous bijection from a compact space onto a Hausdorff space. Hence it is a homeomorphism, and its inverse is continuous.

Conversely, if $D$ is weakly compact and
$\operatorname{id}:(D,w)\to(D,\|\cdot\|)$ is continuous, then $D$ is
norm-compact as the continuous image of a compact space.
\end{proof}

The preceding characterization suggests isolating the class of normed
spaces for which the additional weak-to-norm continuity condition is
automatic on the unit sphere. This leads naturally to the Kadec
property, which goes back to the classical work of Kadec on weak and
norm convergence \cite{Kadec1958}; see also \cite{Raja1999} for a
modern treatment.

\begin{definition}
\label{def:kadec}
A normed space $X$ is said to have a \emph{Kadec norm}, or simply to be
\emph{Kadec}, if the weak and norm topologies coincide on the unit
sphere $S_X$. Equivalently, the identity mapping
\[
\operatorname{id}:(S_X,w)\longrightarrow(S_X,\|\cdot\|)
\]
is continuous.
\end{definition}

For Kadec spaces, the additional continuity condition in
Proposition~\ref{prop:directional-compactness-weak} is automatic.
Consequently, directional compactness admits a purely weak
characterization.

\begin{corollary}
\label{cor:kadec-directional-compactness}
Let $X$ be a Kadec normed space and let
$A\subset X\setminus\{0_X\}$. Then the following assertions are
equivalent:
\begin{enumerate}
\item[(i)] $A$ is directionally compact;
\item[(ii)] $A$ is directionally weakly compact.
\end{enumerate}
\end{corollary}

\begin{proof}
Since $\Pi_{S_X}(A)\subset S_X$, the Kadec property implies that the identity mapping
\[
\operatorname{id}:
(\Pi_{S_X}(A),w)
\longrightarrow
(\Pi_{S_X}(A),\|\cdot\|)
\]
is continuous. The conclusion follows from
Proposition~\ref{prop:directional-compactness-weak}.
\end{proof}

\begin{remark}\label{rem:kadec-examples}
The Kadec property is automatic in every finite-dimensional normed space.
Important infinite-dimensional examples are provided by uniformly convex
spaces. In particular, Hilbert spaces and the classical spaces $\ell^p$ and
$L^p(\mu)$, $1<p<\infty$, endowed with their usual norms, are Kadec; see,
e.g., \cite{Megginson1998,Clarkson1936}.
\end{remark}

Combining the preceding equivalence with
Theorem~\ref{thm:finite-separation-directionally-compact} and
Corollary~\ref{cor:single-separator-directionally-compact}
yields the following weak-compactness version of finite separation in
Kadec spaces.

\begin{corollary}
\label{cor:kadec-directionally-weak-compact-separation}
Let $X$ be a Kadec normed space, let $\cF$ be directionally fine, let
$K\subset X$ be a closed cone, and let
$A\subset X\setminus\{0_X\}$ be directionally weakly compact with
$A\cap K=\varnothing$. Then there exist $C_1,\ldots,C_n\in\cF$ and
adapted defining functions $\varphi_1,\ldots,\varphi_n$ such that
$A\subset\bigcup_{i=1}^n\intt C_i$ and
$K\cap C_i=\{0_X\}$ for every $i$. Moreover, setting $\Phi:=2\max_{1\leq i\leq n}\varphi_i$, the function
$\Phi$ is continuous and positively homogeneous,
$\{\Phi\geq0\}=\bigcup_{i=1}^nC_i$,
$\{\Phi\geq0\}\cap K=\{0_X\}$, and
$\Phi(k)<0<\|a\|<\Phi(a)$ for every
$k\in K\setminus\{0_X\}$ and $a\in A$.
\end{corollary}

\begin{proof}
By Corollary~\ref{cor:kadec-directional-compactness}, $A$ is
directionally compact. The conclusions now follow from
Theorem~\ref{thm:finite-separation-directionally-compact} and
Corollary~\ref{cor:single-separator-directionally-compact}.
\end{proof}

Here one must distinguish weak compactness of $A$ from directional weak
compactness. Even when $0_X\notin A$, the former does not imply the latter:
the normalization mapping need not preserve weak compactness. The following
example shows that this failure may occur even in a Hilbert space.
\begin{example}
\label{ex:weak-compactness-radial-projection}
Let $H=\ell_2$, and let $(e_n)_{n\geq1}$ be its canonical basis. Consider $A:=\{e_1\}\cup\{e_1+e_n:n\geq2\}$. Since $(e_n)$ converges weakly to $0_H$, the sequence $(e_1+e_n)$ converges weakly to $e_1$. Hence $A$ is weakly compact. However,
$\Pi_{S_H}(A)=\{e_1\}\cup\{(e_1+e_n)/\sqrt2:n\ge2\}$ is not weakly compact. Indeed, the sequence $((e_1+e_n)/\sqrt2)_{n\geq2}$, converges weakly to $e_1/\sqrt2$, which does not belong to $S_H$, and hence does not belong to
$\Pi_{S_H}(A)$. Thus $\Pi_{S_H}(A)$ is not weakly closed and therefore
cannot be weakly compact. Thus $A$ is weakly compact but not directionally weakly compact. Equivalently, the normalization mapping $x\longmapsto \frac{x}{\|x\|}$ need not preserve weak compactness.
\end{example}

\subsection{Approximation and Separation of Cones}\label{subsec:approximation-separation-cones}
We now specialize the preceding directional arguments to cones, for which compactness of the norm-base provides the natural framework for finite approximation and separation.

We begin by recording some elementary properties of the cones generated by
balls. These cones will play the role of canonical metric conic
neighborhoods.

Let $u\in S_X$ and $0<\varepsilon<1$. We set $G(u,\varepsilon):=\operatorname{cone} B(u,\varepsilon)$.

Notice that, if $\R_+u$ denotes the ray generated by $u$, then $(\mathbb{R}_+u)\cap S_X=\{u\}$, and consequently $G(u,\varepsilon)$ is precisely the
$\varepsilon$-conic neighborhood of the ray $\R_+u$ in the sense of
Kasimbeyli; see \cite{Kasimbeyli2010}.

The next proposition collects the basic geometric properties of these canonical conic neighborhoods and provides the directional estimate used below.

\begin{proposition}\label{prop:propiedades_epsilon_entornos_conicos_de_rayos}
Let $u\in S_X$ and $0<\varepsilon<1$. Then
$G(u,\varepsilon)$ is a closed, convex, solid and pointed cone.
In addition,
\begin{equation}\label{eq:prop_propiedades_epsilon_entornos_conicos_de_rayos}
B(u,\varepsilon)\cap S_X
\subset
G(u,\varepsilon)\cap S_X
=
\left\{
v\in S_X:
d(u,\mathbb R_+v)\le\varepsilon
\right\}
\subset
B(u,2\varepsilon)\cap S_X.
\end{equation}
\end{proposition}

\begin{proof}
Closedness follows from \cite[Lemma 2.1.43(iii)]{Gopfert2023}.
The remaining assertions are immediate or standard, and we therefore only prove the directional estimate. The first inclusion is immediate.

We next prove that
\(G(u,\varepsilon)\cap S_X
=
\{v\in S_X:d(u,\mathbb R_+v)\le\varepsilon\}\).
Let \(v\in G(u,\varepsilon)\cap S_X\). Then there exist \(t>0\) and
\(y\in B(u,\varepsilon)\) such that \(v=ty\). Hence
\(y=t^{-1}v\in\mathbb R_+v\), and therefore
\(d(u,\mathbb R_+v)\le\|u-y\|\le\varepsilon\).

Conversely, let \(v\in S_X\) satisfy
\(d(u,\mathbb R_+v)\le\varepsilon\). Since the ray
\(\mathbb R_+v\) is closed and the function
\(t\mapsto\|u-tv\|\) attains its minimum on \(\mathbb R_+\), there exists
\(t\ge0\) such that
\(\|u-tv\|=d(u,\mathbb R_+v)\le\varepsilon\). Moreover, \(t>0\), since
\(t=0\) would imply \(1=\|u\|\le\varepsilon\), contrary to
\(\varepsilon<1\). Thus \(tv\in B(u,\varepsilon)\), and consequently
\(v=t^{-1}(tv)\in G(u,\varepsilon)\). This proves the equality.

It remains to prove that
\(G(u,\varepsilon)\cap S_X\subset B(u,2\varepsilon)\cap S_X\).
Let \(0\neq x\in G(u,\varepsilon)\) and write \(x=ty\), with \(t>0\)
and \(y\in B(u,\varepsilon)\). Since
\(x/\|x\|=y/\|y\|\), we have  $\|x/\|x\|-u\|
\leq\|y/\|y\|-y\|+\|y-u\|
=|1-\|y\||+\|y-u\|\leq2\varepsilon$. Hence
\(G(u,\varepsilon)\cap S_X\subset B(u,2\varepsilon)\cap S_X\).
\end{proof}

\begin{remark}
\label{rem:directional-topological-interpretation}
Proposition~\ref{prop:propiedades_epsilon_entornos_conicos_de_rayos}
shows that the cones $G(u,\varepsilon)$ provide a canonical metric
description of neighborhoods of directions. Indeed, \eqref{eq:prop_propiedades_epsilon_entornos_conicos_de_rayos}
means that the spherical sections $G(u,\varepsilon)\cap S_X$ determine the
usual norm topology on $S_X$.

This also gives a topological interpretation of
Definition~\ref{def:directionally_fine}. A family $\mathcal F$ is
directionally fine if and only if, for every $u\in S_X$, the sets
$C\cap S_X$, with $C\in\mathcal F$ and
$u\in\operatorname{int}C$, form a neighborhood base at $u$ in $S_X$.
Thus directional fineness may be viewed as a local-base property on
the space of directions, canonically identified with the positive rays
of $X$.

This viewpoint is closely related to the conic dilation studied by
Seeger~\cite{Seeger2025} in the finite-dimensional Euclidean setting,
where a convex cone is enlarged by considering cones generated by
metric neighborhoods of its spherical section.
\end{remark}

To pass from local directional approximations to finite approximations
of an entire cone, we need compactness of its set of directions. This
leads naturally to the following notion.
\begin{definition}\label{defi:locally_compact_cone}
Let $X$ be a normed space and let $K\subset X$ be a cone.
We say that $K$ is \emph{locally compact} if $K$, endowed with the
relative norm topology inherited from $X$, is a locally compact
topological space; equivalently, for every $x\in K$, there exists $r>0$ such that
\(K\cap B(x,r)\)
is compact.
\end{definition}

Locally compact cones have been considered in several contexts in
infinite-dimensional optimization and conic analysis; see, for instance,
\cite{Gowda1989,deLaatVallentin2015}. The relationship between local compactness of cones and compact bases is classical. For a nontrivial cone $K\subset X$, we call
$B_K:=K\cap S_X$ its norm-base (or spherical section). The following elementary characterization of locally compact cones in terms of the norm-base will be useful throughout the paper.

\begin{proposition}\label{prop:local_compactness}
Let $X$ be a normed space and let $K\subset X$ be a nontrivial cone. The following assertions are equivalent:
\begin{enumerate}
    \item[(i)] $K$ is locally compact.

    \item[(ii)] The norm-base $B_K$ is compact.

    \item[(iii)] $K\cap rS_X$ is compact for some $r>0$
    (equivalently, for every $r>0$).

    \item[(iv)] $K\cap rB_X$ is  compact for some $r>0$
    (equivalently, for every $r>0$).
\end{enumerate}
\end{proposition}
\begin{proof}
The equivalence between \textnormal{(ii)} and \textnormal{(iii)} follows
from positive homogeneity, since for every $r>0$ the map
$u\mapsto ru$ is a homeomorphism from $B_K$ onto $K\cap rS_X$.

Assume \textnormal{(ii)}. For every $r>0$,
$K\cap rB_X=\{tu:0\leq t\leq r,\ u\in B_K\}$, which is the continuous
image of the compact set $[0,r]\times B_K$ under $(t,u)\mapsto tu$.
Thus \textnormal{(iv)} holds. Conversely, if $K\cap rB_X$ is compact
for some $r>0$, then $K\cap rS_X=(K\cap rB_X)\cap rS_X$ is compact,
so \textnormal{(iii)} holds.

Assume now \textnormal{(ii)}. Then $K\cap RB_X$ is compact for every
$R>0$. Given $x\in K$, choose $\rho>0$ and
$R>\|x\|+\rho$. The set $K\cap B(x,\rho)$ is a closed subset of the
compact set $K\cap RB_X$ and a neighborhood of $x$ in $K$; hence $K$
is locally compact.

Finally, suppose that $K$ is locally compact. Since $0_X\in K$, there
exists $\delta>0$ such that $K\cap\delta B_X$ is compact. For any
$r\in(0,\delta)$, the set
$K\cap rS_X=(K\cap\delta B_X)\cap rS_X$ is compact. Thus
\textnormal{(iii)} holds.
\end{proof}

\begin{remark}
\label{rem:locally_compact_cone_closed}
Every nontrivial locally compact cone $K\subset X$ is closed. Indeed, by
Proposition~\ref{prop:local_compactness}, its norm-base
\(
B_K=K\cap S_X
\)
is compact, and
\(
K=\mathbb R_+\cdot B_K.
\)
Since $B_K$ is compact and $0_X\notin B_K$, the cone
$\mathbb R_+\cdot B_K$ is closed.
\end{remark}

\begin{remark}\label{rem:locall_cpto_and_directionally_cpto}
Let $K\subset X$ be a nontrivial cone. Then $K$ is locally compact if and only if $K\setminus\{0_X\}$ is directionally compact. Indeed, since $K$ is a cone,
\(
\Pi_{S_X}\bigl(K\setminus\{0_X\}\bigr)
=
K\cap S_X.
\)
Therefore $K\setminus\{0_X\}$ is directionally compact if and only if
the norm-base $K\cap S_X$ is compact. By the preceding proposition,
this is equivalent to the local compactness of $K$.
\end{remark}

In Kadec spaces, Proposition~\ref{prop:local_compactness} admits the
following weak formulation.

\begin{corollary}
\label{cor:kadec-weakly-compact-base}
Let $X$ be a Kadec normed space and let $K\subset X$ be a nontrivial
cone. Then $K$ is locally compact if and only if the norm-base $B_K$ is weakly compact.
\end{corollary}

\begin{proof}
This follows from Remark~\ref{rem:locall_cpto_and_directionally_cpto}
and Corollary~\ref{cor:kadec-directional-compactness}, since $\Pi_{S_X}(K\setminus\{0_X\})=B_K$.
\end{proof}

Consequently, throughout the remainder of this section, whenever $X$
is Kadec, any local compactness assumption on a cone may equivalently
be replaced by weak compactness of its norm-base. By Remark~\ref{rem:kadec-examples}, this applies in particular to Hilbert spaces and to the spaces $\ell^p$ and $L^p(\mu)$,
$1<p<\infty$, with their usual norms.

For $0<\varepsilon<1$, define the
$\varepsilon$-conic neighborhood of $L$ (see \cite[Definition 4.2]{Kasimbeyli2010}) by
\(
L_\varepsilon
:=
\operatorname{cone}(B_L+\varepsilon B_X).
\)
Equivalently,
\(
L_\varepsilon
=
\bigcup_{u\in B_L}G(u,\varepsilon).
\)

The following result provides a finite outer approximation of a locally compact cone by cones from a directionally fine family, together with quantitative directional control.

\begin{theorem}\label{thm:finite-outer-approximation}
Let $\mathcal F$ be a directionally fine family of closed cones in $X$ and let
$L\subset X$ be a locally compact cone. For every $\varepsilon\in(0,1)$ there exist
$C_1,\ldots,C_m\in\mathcal F$ such that $D_L^\varepsilon:=\bigcup_{i=1}^m C_i$ is a closed cone and
\[
L\setminus\{0_X\}\subset \bigcup_{i=1}^m\operatorname{int}C_{i}\subset\operatorname{int}D_L^\varepsilon,\qquad
L\subset D_L^\varepsilon\subset L_\varepsilon,
\]
and $D_L^\varepsilon\cap S_X\subset\{x\in S_X:d(x,B_L)\leq 2\varepsilon\}.$
\end{theorem}

\begin{proof}
Fix $\varepsilon\in(0,1)$. For each $u\in B_L$, choose
$C_u\in\mathcal F$ such that $u\in\operatorname{int}C_u\subset C_u\subset G(u,\varepsilon) \subset L_\varepsilon$. Since $B_L$ is compact, there exist $u_1,\ldots,u_m\in B_L$ such that $B_L\subset\bigcup_{i=1}^m\operatorname{int}C_{u_i}$. Then, $L\setminus\{0_X\}\subset\bigcup_{i=1}^m\operatorname{int}C_{u_i}$. Set $C_i:=C_{u_i}$ and $D_L^\varepsilon:=\bigcup_{i=1}^m C_i$. Then $D_L^\varepsilon$ is a closed
cone and $D_L^\varepsilon\subset L_\varepsilon$. Moreover, if $x\in L\setminus\{0_X\}$,
then $x/\|x\|\in\operatorname{int}C_i$ for some $i$, and hence
$x\in\operatorname{int}C_i\subset\operatorname{int}D_L^\varepsilon$. Consequently,
$L\subset D_L^\varepsilon$. Finally, if $x\in D_L^\varepsilon\cap S_X$, then
$x\in C_i\cap S_X\subset G(u_i,\varepsilon)\cap S_X$ for some $i$.
Proposition~\ref{prop:propiedades_epsilon_entornos_conicos_de_rayos}
therefore gives $d(x,B_L)\leq\|x-u_i\|\leq2\varepsilon$, which completes the proof.
\end{proof}
The preceding finite outer approximation also yields a quantitative
Hausdorff estimate on the corresponding sets of directions and on bounded
sections.
\begin{corollary}
\label{cor:finite-hausdorff-approximation}
Under the assumptions of Theorem~\ref{thm:finite-outer-approximation},
for each $\varepsilon\in(0,1)$ one can choose a finite union
$D_L^\varepsilon$ of members of $\cF$ such that
\[
d_H(D_L^\varepsilon\cap S_X,B_L)\leq2\varepsilon,
\qquad
d_H(D_L^\varepsilon\cap RB_X,L\cap RB_X)\leq2R\varepsilon,
\quad \forall R>0.
\]
Consequently, $D_L^\varepsilon\cap S_X$ converges to $B_L$ and, for every
$R>0$, $D_L^\varepsilon\cap RB_X$ converges to $L\cap RB_X$ in Hausdorff
distance as $\varepsilon\downarrow0$.
\end{corollary}

\begin{proof}
Choose $D_L^\varepsilon$ as in
Theorem~\ref{thm:finite-outer-approximation}. Since
$B_L\subset D_L^\varepsilon\cap S_X$ and
$d(x,B_L)\leq2\varepsilon$ for every
$x\in D_L^\varepsilon\cap S_X$, the first estimate follows.

Now let $0_X\neq x\in D_L^\varepsilon\cap RB_X$ and write
$x=tv$, where $t=\|x\|\leq R$ and
$v\in D_L^\varepsilon\cap S_X$. Choose $u\in B_L$ with
$\|v-u\|\leq2\varepsilon$. Then $tu\in L\cap RB_X$ and
$d(x,L\cap RB_X)\leq t\|v-u\|\leq2R\varepsilon$. Since $L\subset D_L^\varepsilon$, the second estimate follows.
\end{proof}

A finite separation conclusion could also be obtained directly from
Theorem~\ref{thm:finite-separation-directionally-compact}, since
Remark~\ref{rem:locall_cpto_and_directionally_cpto} shows that
$L\setminus\{0_X\}$ is directionally compact. The preceding approximation
results yield, however, a stronger geometric statement: the separating
finite union can be chosen inside any sufficiently small prescribed conic
neighborhood $L_\varepsilon$ of $L$. In addition,
Corollary~\ref{cor:finite-hausdorff-approximation} provides quantitative
Hausdorff control of the corresponding directions and bounded sections.

\begin{theorem}
\label{thm:two-cones-finite-separation}
Let $\cF$ be a directionally fine family of closed cones, and let
$L,K\subset X$ be cones such that $L\cap K=\{0_X\}$, with $K$ closed
and $L$ locally compact. Set $d:=d(B_L,K)$. Then $d>0$, and for every
$0<\varepsilon<\min\{1,d\}$ there exists a closed cone $D_L^\varepsilon$, which is a
finite union of members of $\cF$, such that
\[
L\setminus\{0_X\} \subset  \operatorname{int}D_L^\varepsilon,\quad
L\subset D_L^\varepsilon\subset L_\varepsilon,\qquad
D_L^\varepsilon\cap K=\{0_X\}.
\]
\end{theorem}

\begin{proof}
Since $L\cap K=\{0_X\}$, one has $B_L\cap K=\varnothing$. As $B_L$ is
compact and $K$ is closed, $d:=d(B_L,K)>0$. Fix
$0<\varepsilon<\min\{1,d\}$. Then $L_\varepsilon\cap K=\{0_X\}$:
indeed, if $0_X\neq x=t(u+\varepsilon b)\in L_\varepsilon\cap K$,
with $t>0$, $u\in B_L$ and $b\in B_X$, then
$u+\varepsilon b=x/t\in K$, and hence $d(u,K)\leq\varepsilon<d$, a
contradiction.

By Theorem~\ref{thm:finite-outer-approximation}, there exists a closed
cone $D_L^\varepsilon$, which is a finite union
$D_L^\varepsilon=\bigcup_{i=1}^m C_i$ of members of $\cF$, such that
$L\setminus\{0_X\}\subset\operatorname{int}D_L^\varepsilon$ and
$L\subset D_L^\varepsilon\subset L_\varepsilon$. Moreover, its construction yields $B_L\subset\bigcup_{i=1}^m\operatorname{int}C_i$. Hence
$D_L^\varepsilon\cap K\subset L_\varepsilon\cap K=\{0_X\}$.
\end{proof}

The finite family underlying the cone $D_L^\varepsilon$, in the preceding result, can be compressed into a single continuous positively homogeneous separator.

\begin{corollary}
\label{cor:max-two-cones}
Under the assumptions of Theorem~\ref{thm:two-cones-finite-separation},
there exist a continuous positively homogeneous function
$\Phi:X\to\mathbb R$ and $\gamma_L>0$ such that $|\Phi(x)|\leq\|x\|$ for every $x\in X$ and
\[
\Phi(l)\geq\gamma_L\|l\|>0>\Phi(k),
\quad \forall l\in L\setminus\{0_X\},\ \forall k\in K\setminus\{0_X\}.
\]
\end{corollary}

\begin{proof}
Take \(D_L^\varepsilon=\bigcup_{i=1}^m C_i\) as in
Theorem~\ref{thm:two-cones-finite-separation}, where
$u_i\in B_L\cap\operatorname{int}C_i$. For each $i$, let $\varphi_i$ be an adapted defining function for $C_i$
at $u_i$ provided by Lemma~\ref{lem:adapted_defining_function}. Then
\(
C_i=\{\varphi_i\ge0\},\quad
\operatorname{int}C_i=\{\varphi_i>0\},
\quad
|\varphi_i(x)|\le\|x\|,
\quad \forall x\in X.
\)
Set
$\Phi:=\max_{1\leq i\leq m}\varphi_i$. Then $\Phi$ is continuous and
positively homogeneous, $D_L^\varepsilon=\{\Phi\geq0\}$, and
$|\Phi(x)|\leq\|x\|$. Since
$L\setminus\{0_X\}\subset\bigcup_{i=1}^m\operatorname{int}C_i$,
one has $\Phi>0$ on $L\setminus\{0_X\}$. Moreover,
$D_L^\varepsilon\cap K=\{0_X\}$ implies $\varphi_i(k)<0$ for every $i$
and every $k\in K\setminus\{0_X\}$, and hence $\Phi(k)<0$.

Finally, $B_L$ is compact, so
$\gamma_L:=\min_{u\in B_L}\Phi(u)>0$. For
$l\in L\setminus\{0_X\}$, positive homogeneity gives
$\Phi(l)=\|l\|\Phi(l/\|l\|)\geq\gamma_L\|l\|$.
\end{proof}

If both cones are locally compact, the preceding results yield a
symmetric finite separation.

\begin{theorem}
\label{thm:symmetric-finite-separation}
Let $\mathcal F$ be a directionally fine family of closed cones, and let
$L,K\subset X$ be locally compact cones such that
$L\cap K=\{0_X\}$. Set $d:=d(B_L,B_K)>0$. Then, for every $0<\varepsilon<\min\{1,d/4\}$, there exist closed cones
$D_L^\varepsilon,D_K^\varepsilon$, each a finite union of members of
$\mathcal F$, such that
\[
L\setminus\{0_X\}\subset\intt D_L^\varepsilon,\qquad
K\setminus\{0_X\}\subset\intt D_K^\varepsilon,\qquad
D_L^\varepsilon\cap D_K^\varepsilon=\{0_X\},
\]
with
$L\subset D_L^\varepsilon\subset L_\varepsilon$ and
$K\subset D_K^\varepsilon\subset K_\varepsilon$.
Moreover,
$D_L^\varepsilon\cap S_X
 \subset\{x\in S_X:d(x,B_L)\leq2\varepsilon\}$
and
$D_K^\varepsilon\cap S_X
 \subset\{x\in S_X:d(x,B_K)\leq2\varepsilon\}$.

In addition, there exist a continuous positively homogeneous function
$\Phi:X\to\R$ and constants $\gamma_L,\gamma_K>0$ such that
$|\Phi(x)|\leq\|x\|$ for every $x\in X$ and
\[
\Phi(l)\geq\gamma_L\|l\|,
\qquad
\Phi(k)\leq-\gamma_K\|k\|,
\qquad \forall l\in L,\quad \forall k\in K.
\]
\end{theorem}

\begin{proof}
Since $B_L$ and $B_K$ are compact and disjoint, let
$d:=d(B_L,B_K)>0$ and fix
$0<\varepsilon<\min\{1,d/4\}$. By
Theorem~\ref{thm:finite-outer-approximation}, there exist closed finite
unions $D_L^\varepsilon,D_K^\varepsilon$ satisfying the stated
interior and approximation properties. If
$0_X\neq x\in D_L^\varepsilon\cap D_K^\varepsilon$ and
$v:=x/\|x\|$, there exist $u\in B_L$ and $w\in B_K$ such that
$\|v-u\|\leq2\varepsilon$ and $\|v-w\|\leq2\varepsilon$. Hence
$d\leq\|u-w\|\leq\|u-v\|+\|v-w\|\leq4\varepsilon<d$, a contradiction.

By Corollary~\ref{cor:max-two-cones}, there exist a continuous
positively homogeneous function $\Phi$ and $\gamma_L>0$ such that
$\Phi(l)\geq\gamma_L\|l\|$ for every $l\in L$,
$\Phi<0$ on $K\setminus\{0_X\}$, and $|\Phi(x)|\leq\|x\|$.
Since $B_K$ is compact,
$\gamma_K:=-\max_{v\in B_K}\Phi(v)>0$. Thus, for
$k\in K\setminus\{0_X\}$,
$\Phi(k)=\|k\|\Phi(k/\|k\|)\leq-\gamma_K\|k\|$; for $k=0_X$ the
inequality is immediate.
\end{proof}

\begin{remark}\label{rem:finite-dimensional-cones}
If $X$ is finite-dimensional, every closed cone $L\subset X$ is locally
compact. Indeed, $B_L=L\cap S_X$ is a closed subset of the compact set
$S_X$, and hence is compact. Consequently, all the approximation and
finite-separation results of this subsection apply to arbitrary closed
cones in finite-dimensional spaces, without any additional compactness
assumption. In particular,
Theorems~\ref{thm:finite-outer-approximation},
\ref{thm:two-cones-finite-separation}, and
\ref{thm:symmetric-finite-separation}, together with Corollaries \ref{cor:finite-hausdorff-approximation} and \ref{cor:max-two-cones}, hold for closed cones under their respective disjointness assumptions.
\end{remark}

The finite character of our construction is essential. Even
in the plane, there are pairs of cones that can be separated by a
finite union of narrow cones, but cannot be separated by any single
convex cone.
\begin{example}\label{ex:no-single-convex-cone}
Let $X=\mathbb R^2$ with any norm, $e_1=(1,0)$, and
$u_\pm=(\cos\theta,\pm\sin\theta)$ for some
$0<\theta<\pi/2$. Set $L:=\mathbb R_+u_+\cup\mathbb R_+u_-$ and
$K:=\mathbb R_+e_1$. Then $L\cap K=\{0_X\}$, but no convex cone $C$
can satisfy $L\subset C$ and $C\cap K=\{0_X\}$. Indeed, if $u_\pm\in C$, then
$u_++u_-=2\cos\theta\,e_1\in C$, hence $K\subset C$.

In contrast, if $\mathcal F$ is directionally fine, one can choose
$C_\pm\in\mathcal F$ sufficiently narrow around the directions
$u_\pm/\|u_\pm\|$ so that $L\setminus\{0_X\}\subset\operatorname{int}(C_+\cup C_-)$ and $(C_+\cup C_-)\cap K=\{0_X\}$.

If $\varphi_\pm$ are adapted defining functions for $C_\pm$, then
$\Phi:=\max\{\varphi_+,\varphi_-\}$ satisfies
$\Phi(l)\geq\gamma_L\|l\|$ and
$\Phi(k)\leq-\gamma_K\|k\|$ for some $\gamma_L,\gamma_K>0$ and all
$l\in L$, $k\in K$.
\end{example}

\begin{remark}
\label{rem:finite-family-essential}
Example~\ref{ex:no-single-convex-cone} shows that the passage from a
single separator to a finite family is not merely a technical artifact
of the compactness argument.

First, finite unions genuinely enlarge the class of pairs of cones
that can be separated. Second, the max-type scalarization
\[
\Phi=\max_{1\leq i\leq m}\varphi_i,
\]
is therefore intrinsic to the construction. Even when each individual
$\varphi_i$ is superlinear, their maximum need not be superlinear.
Thus the possible loss of superlinearity at the global level reflects
the geometry of the problem rather than a weakness of the proof.

This also distinguishes the present approach from separation results
based on a single Bishop--Phelps cone, where the separating cone is
convex and the relevant hypotheses are naturally expressed in terms of
convexifications of norm-bases and suitable strict separation
properties; see, for instance, \cite{Kasimbeyli2010,GarciaCastanoEtAl2025}.
\end{remark}

\section{The Bishop--Phelps Family: Directional Fineness and Max-Type Separation}
The aim of this section is to determine when the classical family of Bishop--Phelps cones fits the abstract framework developed above. We show that this happens precisely when every point of the unit sphere is a denting point of the unit ball. In Banach spaces, this condition is equivalently characterized by rotundity together with the Kadec property. We then record the corresponding finite separation and approximation consequences as direct applications of the results of Section~\ref{sec:separation-approximation}.

We now consider the classical family of Bishop--Phelps cones. For
$f\in S_{X^*}$ and $0<\alpha<1$, set $C(f,\alpha)
:=
\{x\in X:f(x)\geq \alpha\|x\|\}$, and define
\[
\mathcal F_{\mathrm{BP}}
:=
\{C(f,\alpha):f\in S_{X^*},\ 0<\alpha<1\}.
\]

These classical Bishop--Phelps cones go back to Phelps and have been
widely used in vector optimization and nonlinear separation; see, e.g.,
\cite{Phelps1974,Jahn2009,Kasimbeyli2010,HaJahn2017}. Every member of $\mathcal F_{\mathrm{BP}}$ is a closed convex cone. However, $\mathcal F_{\mathrm{BP}}$
need not be directionally fine for an arbitrary norm.

Recall that $u\in B_X$ is a denting point of $B_X$ if, for every
$\varepsilon>0$, there exist $f\in S_{X^*}$ and $\gamma>0$ such that
$u\in S(B_X,f,\gamma)$ and $\operatorname{diam}S(B_X,f,\gamma)<\varepsilon$,
where $S(B_X,f,\gamma):=\{x\in B_X:f(x)>1-\gamma\}$ is the corresponding slice of $B_X$; see, for instance, \cite{LinLinTroyanski1988,GarciaCastanoMelguizoMontesinos2015}. Denting points also occur naturally in density results for proper minimal points; see, e.g., \cite{Gong1995}.

The following result identifies the precise geometric condition under which
the family of Bishop--Phelps cones is directionally fine.
\begin{theorem}
\label{thm:BP-directionally-fine}
The family $\mathcal F_{\mathrm{BP}}$ is directionally fine if and only if
every point of $S_X$ is a denting point of $B_X$.
Consequently, whenever this condition holds,
$\mathcal F_{\mathrm{BP}}$ is a convex directionally fine family.
\end{theorem}

\begin{proof}
Assume first that every $u\in S_X$ is a denting point of $B_X$.
Fix $u\in S_X$ and $\varepsilon\in(0,1)$, and choose
$g_\varepsilon\in S_{X^*}$ and $\gamma>0$ such that
$u\in S:=\{x\in B_X:g_\varepsilon(x)>1-\gamma\}$ and
$\operatorname{diam}S<\min\{\varepsilon,1/2\}$. Since
$\sup_S g_\varepsilon=1$, one has
$g_\varepsilon(u)\geq1-\operatorname{diam}S>1/2$. Choose
$\max\{0,1-\gamma\}<\alpha_\varepsilon<g_\varepsilon(u)$.
Then
$u\in\operatorname{int}C(g_\varepsilon,\alpha_\varepsilon)$ and
$C(g_\varepsilon,\alpha_\varepsilon)\cap S_X
\subset S\subset B(u,\varepsilon)$, whence
$C(g_\varepsilon,\alpha_\varepsilon)
\subset\operatorname{cone}B(u,\varepsilon)$. Thus
$\mathcal F_{\mathrm{BP}}$ is directionally fine.

Conversely, assume that $\mathcal F_{\mathrm{BP}}$ is directionally
fine. Fix $u\in S_X$ and $\eta>0$, and choose
$0<\delta<\min\{1,\eta/10\}$. There exist
$g_\delta\in S_{X^*}$ and $\alpha_\delta\in(0,1)$ such that
$u\in\operatorname{int}C(g_\delta,\alpha_\delta)
\subset C(g_\delta,\alpha_\delta)\subset G(u,\delta)$. Hence, by
Proposition~\ref{prop:propiedades_epsilon_entornos_conicos_de_rayos},
$C(g_\delta,\alpha_\delta)\cap S_X\subset B(u,2\delta)$.
Approximating the supremum of $g_\delta$ on $S_X$ gives
$g_\delta(u)\geq1-2\delta$. Choose
$\max\{\alpha_\delta,1-3\delta\}<\beta<g_\delta(u)$ and set
$S_\beta:=\{x\in B_X:g_\delta(x)>\beta\}$.

If $x\in S_\beta$, then $x\in C(g_\delta,\alpha_\delta)$,
$\|x\|>1-3\delta$, and
$x/\|x\|\in C(g_\delta,\alpha_\delta)\cap S_X$. Therefore $\|x-u\|
\leq 1-\|x\|
   +\left\|\frac{x}{\|x\|}-u\right\|
<5\delta$. Thus $S_\beta$ is a slice containing $u$ with
$\operatorname{diam}S_\beta\leq10\delta<\eta$, so $u$ is a denting
point of $B_X$.
\end{proof}

\begin{remark}
\label{rem:BP-property-G}
The geometric condition in
Theorem~\ref{thm:BP-directionally-fine} is the classical
property~$(G)$ introduced by Fan and Glicksberg
\cite{FanGlicksberg1958}: every point of $S_X$ is a denting point of
$B_X$. Hence, for an arbitrary normed space, $\mathcal F_{\mathrm{BP}}$ is
directionally fine if and only if $X$ has property~(G). Every locally uniformly rotund (LUR) normed space has property~$(G)$. Consequently, the
Bishop--Phelps family is directionally fine in every uniformly convex
space and, in particular, in Hilbert spaces and in the classical spaces
$\ell^p$ and $L^p(\mu)$, $1<p<\infty$, endowed with their usual norms;
see, e.g., \cite{Megginson1998,Clarkson1936}.

In the Banach-space setting, property~$(G)$ is equivalent to $X$ being
rotund (strictly convex) and having the Kadec property, \cite{LinLinTroyanski1986}. Thus the geometric condition appearing here
is closely related to the Kadec property used in
Section~\ref{sec:separation-approximation} to pass from weak directional
compactness to directional compactness.

Finally, in finite-dimensional normed spaces the Kadec property is
automatic. Therefore, in finite-dimensional normed spaces,
$\mathcal F_{\mathrm{BP}}$ is directionally fine if and only if $X$ is
strictly convex.
\end{remark}

Combining the preceding characterization of directionally fine
Bishop--Phelps families with the max-type separation result yields the
following explicit scalarization.
\begin{corollary}\label{cor:explicit-BP-max-scalarization}
Assume that every point of $S_X$ is a denting point of $B_X$.
Let $K\subset X$ be a closed cone and let
$A\subset X\setminus\{0_X\}$ be directionally compact with
$A\cap K=\varnothing$. Then there exist $m\in\mathbb N$, $g_1,\ldots,g_m\in S_{X^*}$ and
$\alpha_1,\ldots,\alpha_m\in(0,1)$ such that, setting
$C_i:=C(g_i,\alpha_i)=\{x:g_i(x)\geq\alpha_i\|x\|\}$ and
$D:=\bigcup_{i=1}^mC_i$, one has
$A\subset\bigcup_{i=1}^m\operatorname{int}C_i\subset\operatorname{int}D$
and $D\cap K=\{0_X\}$.

Moreover, the function
\[
\Phi_{\mathrm{BP}}(x)
:=\frac12\max_{1\leq i\leq m}
\bigl(g_i(x)-\alpha_i\|x\|\bigr),
\]
is continuous and positively homogeneous, satisfies
$D=\{\Phi_{\mathrm{BP}}\geq0\}$ and
$|\Phi_{\mathrm{BP}}(x)|\leq\|x\|$ for every $x\in X$, and there exists
$\eta>0$ such that
\[
\Phi_{\mathrm{BP}}(a)\geq\eta\|a\|>0,
\quad \forall a\in A,\qquad
\Phi_{\mathrm{BP}}(k)\leq0,
\quad \forall k\in K,
\]
with $\Phi_{\mathrm{BP}}(k)<0$ for every
$k\in K\setminus\{0_X\}$.
\end{corollary}
\begin{proof}
By Theorem~\ref{thm:BP-directionally-fine},
$\mathcal F_{\mathrm{BP}}$ is directionally fine. Hence
Theorem~\ref{thm:finite-separation-directionally-compact}
provides Bishop--Phelps cones $C_1,\ldots,C_m$ such that
$A\subset\bigcup_{i=1}^m\operatorname{int}C_i$ and
$C_i\cap K=\{0_X\}$ for every $i$. Consequently,
$D=\bigcup_{i=1}^m C_i$ is a closed cone and
$D\cap K=\{0_X\}$. For each $i$, set $p_i(x):=g_i(x)-\alpha_i\|x\|$. Then
$C_i=\{p_i\geq0\}$ and $\operatorname{int}C_i=\{p_i>0\}$, whence $D=\left\{\max_{1\leq i\leq m}p_i\geq0\right\}
=\{\Phi_{\mathrm{BP}}\geq0\}$. Set $M_A:=\Pi_{S_X}(A)$. Since the $C_i$ are cones, $M_A\subset\bigcup_{i=1}^m\operatorname{int}C_i$, and therefore
$\Phi_{\mathrm{BP}}>0$ on $M_A$. By directional compactness,
$M_A$ is compact, so
$\eta:=\min_{u\in M_A}\Phi_{\mathrm{BP}}(u)>0$. Hence, for every
$a\in A$, $\Phi_{\mathrm{BP}}(a)
=\|a\|\Phi_{\mathrm{BP}}(a/\|a\|)
\geq\eta\|a\|$. If $0_X\neq k\in K$, then $k\notin C_i$ for every $i$, and thus
$p_i(k)<0$ for every $i$. Hence $\Phi_{\mathrm{BP}}(k)<0$, while
$\Phi_{\mathrm{BP}}(0_X)=0$. Finally, $|p_i(x)|\leq |g_i(x)|+\alpha_i\|x\|\leq2\|x\|$
for every $i$ and $x\in X$. Hence
$|\Phi_{\mathrm{BP}}(x)|
\leq\frac12\max_{1\leq i\leq m}|p_i(x)|
\leq\|x\|$, which completes the proof.
\end{proof}
The preceding corollary provides an explicit Bishop--Phelps realization
of finite separation for directionally compact sets. For locally compact
cones, the approximation results of Subsection~\ref{subsec:approximation-separation-cones}
yield more: the finite union may be chosen as an arbitrarily fine outer
approximation, with quantitative Hausdorff control, while preserving
separation from a second cone and, when both cones are locally compact,
allowing a symmetric approximation.
\begin{corollary}
\label{cor:BP-approximation-separation}
Assume that every point of $S_X$ is a denting point of $B_X$.

\begin{enumerate}
\item[\rm (i)]
Let $L\subset X$ be a locally compact cone. For every
$\varepsilon\in(0,1)$ there exist Bishop--Phelps cones
$C_1,\ldots,C_m$ such that $D_L^\varepsilon:=\bigcup_{i=1}^m C_i$ is
closed,
\[
L\setminus\{0_X\}\subset\operatorname{int}D_L^\varepsilon,
\qquad
L\subset D_L^\varepsilon\subset L_\varepsilon,
\]
and
\[
d_H(D_L^\varepsilon\cap S_X,B_L)\le2\varepsilon,
\qquad
d_H(D_L^\varepsilon\cap RB_X,L\cap RB_X)\le2R\varepsilon
\quad(R>0).
\]
Hence $D_L^\varepsilon\cap S_X\to B_L$ and
$D_L^\varepsilon\cap RB_X\to L\cap RB_X$ in Hausdorff distance as
$\varepsilon\downarrow0$, for every $R>0$.

\item[\rm (ii)]
Let $L,K\subset X$ be cones with $L\cap K=\{0_X\}$, where $L$ is
locally compact and $K$ is closed, and put $d:=d(B_L,K)>0$. For every
$0<\varepsilon<\min\{1,d\}$, the cones in {\rm (i)} may be chosen so that
$D_L^\varepsilon\cap K=\{0_X\}$. Writing
$C_i=C(g_i,\alpha_i)$, the function
\[
\Phi_{\mathrm{BP},\varepsilon}(x)
:=
\frac12\max_{1\le i\le m}
\bigl(g_i(x)-\alpha_i\|x\|\bigr)
\]
is continuous and positively homogeneous, with
\[
D_L^\varepsilon=\{\Phi_{\mathrm{BP},\varepsilon}\ge0\},
\qquad
|\Phi_{\mathrm{BP},\varepsilon}|\le\|\cdot\|.
\]
Moreover, for some $\eta_\varepsilon>0$,
\[
\Phi_{\mathrm{BP},\varepsilon}(l)
\ge\eta_\varepsilon\|l\|>0
\quad(l\in L\setminus\{0_X\}),
\qquad
\Phi_{\mathrm{BP},\varepsilon}(k)<0
\quad(k\in K\setminus\{0_X\}).
\]

\item[\rm (iii)]
Let $L,K\subset X$ be locally compact cones with
$L\cap K=\{0_X\}$, and put $d:=d(B_L,B_K)>0$. For every
$0<\varepsilon<\min\{1,d/4\}$ there exist finite unions of
Bishop--Phelps cones $D_L^\varepsilon,D_K^\varepsilon$ such that
\[
L\setminus\{0_X\}\subset\operatorname{int}D_L^\varepsilon,
\qquad
K\setminus\{0_X\}\subset\operatorname{int}D_K^\varepsilon,
\]
\[
L\subset D_L^\varepsilon\subset L_\varepsilon,
\qquad
K\subset D_K^\varepsilon\subset K_\varepsilon,
\qquad
D_L^\varepsilon\cap D_K^\varepsilon=\{0_X\},
\]
and
\[
\max\!\left\{
d_H(D_L^\varepsilon\cap S_X,B_L),
d_H(D_K^\varepsilon\cap S_X,B_K)
\right\}\le2\varepsilon,
\]
\[
\begin{aligned}
\max\!\bigl\{
&d_H(D_L^\varepsilon\cap RB_X,L\cap RB_X),\\
&d_H(D_K^\varepsilon\cap RB_X,K\cap RB_X)
\bigr\}
\le 2R\varepsilon,
\qquad R>0.
\end{aligned}
\]
\end{enumerate}
\end{corollary}

\begin{proof}
By Theorem~\ref{thm:BP-directionally-fine},
$\mathcal F_{\mathrm{BP}}$ is directionally fine. Part~{\rm (i)} follows
from Theorem~\ref{thm:finite-outer-approximation} and
Corollary~\ref{cor:finite-hausdorff-approximation}.

Part~{\rm (ii)} follows from
Theorem~\ref{thm:two-cones-finite-separation} and
Corollary~\ref{cor:finite-hausdorff-approximation}. Moreover, by the
construction in Theorem~\ref{thm:two-cones-finite-separation}, the
cones $C_1,\ldots,C_m$ may be chosen so that $B_L\subset\bigcup_{i=1}^m\operatorname{int}C_i$. Hence $\Phi_{\mathrm{BP},\varepsilon}>0$ on the compact set $B_L$, so $\eta_\varepsilon
:=
\min_{u\in B_L}\Phi_{\mathrm{BP},\varepsilon}(u)>0$.
Positive homogeneity gives
$\Phi_{\mathrm{BP},\varepsilon}(l)\ge
\eta_\varepsilon\|l\|$ for $l\in L\setminus\{0_X\}$; the remaining
assertions concerning $\Phi_{\mathrm{BP},\varepsilon}$ follow exactly
as in Corollary~\ref{cor:explicit-BP-max-scalarization}.

For {\rm (iii)}, apply part~{\rm (i)} separately to $L$ and $K$.
If $0_X\ne x\in D_L^\varepsilon\cap D_K^\varepsilon$ and
$v:=x/\|x\|$, the Hausdorff estimates give $u\in B_L$ and $w\in B_K$
with $\|v-u\|\le2\varepsilon$ and $\|v-w\|\le2\varepsilon$. Hence $d\le\|u-w\|
\le\|u-v\|+\|v-w\|
\le4\varepsilon<d$, a contradiction.
\end{proof}

\section{Axial construction principles for directionally fine families}
\label{sec:menu_directionally_fine_families}
The preceding section singles out the classical Bishop--Phelps family:
by Theorem~\ref{thm:BP-directionally-fine}, it is directionally fine
precisely when every point of $S_X$ is a denting point of $B_X$.
The purpose of the present section is complementary. We first show that
every closed convex cone which is already localized around a prescribed
direction admits an axial representation generated by a transversally
coercive deviation. In particular, this gives a direct link with the
Bishop--Phelps cones selected in Section~4. We then use transversal
coercivity in the opposite direction, as a general construction principle
producing directionally fine families in arbitrary normed spaces, without
any rotundity or dentability assumption on the norm. Finally, we illustrate
how the explicit structure of one of the resulting families can be combined
with the finite outer approximation theory of Section~3 to obtain
penalization, recovery, and convergence results for optimization over
locally compact cones.

Although Lemma~\ref{lem:adapted_defining_function} provides an adapted
defining function for every closed cone selected by directional fineness,
the families considered below come with explicit defining functions arising
directly from their analytic description. Such functions retain the
geometric data of the construction and are therefore preferable when
further applications are envisaged, as illustrated in
Subsection~\ref{subsec:UAD-optimization}.

\begin{definition}
\label{def:norming-pair-axial}
A pair $(u,f)\in S_X\times S_{X^*}$ is called a \emph{norming pair} if
$f(u)=1$. Associated with $(u,f)$, define $Q_{u,f}:=I-u\otimes f$, $Q_{u,f}x=x-f(x)u$.
Then $x=f(x)u+Q_{u,f}x$, $Q_{u,f}x\in\ker f$, $Q_{u,f}^2=Q_{u,f}$, $\ker Q_{u,f}=\mathbb Ru$, $\operatorname{ran}Q_{u,f}=\ker f$, $\|Q_{u,f}\|\leq2$. Thus $Q_{u,f}$ is the projection onto $\ker f$ along $\mathbb Ru$.

If $f(x)>0$, then $x/f(x)$ is the unique representative of the positive
ray generated by $x$ in the affine section $\{f=1\}$, and $\left\|
\frac{x}{f(x)}-u
\right\|
=
\frac{\|Q_{u,f}x\|}{f(x)}$. Hence the quotient on the right measures the deviation of the direction
of $x$ from $u$ in this affine section.
\end{definition}

This observation suggests controlling the transversal component
$Q_{u,f}x$ through a more general positively homogeneous deviation.

\begin{definition}\label{def:transversal-coercivity}
Let $(u,f)$ be a norming pair. A continuous positively homogeneous
mapping $D_{u,f}:X\longrightarrow[0,+\infty)$ is said to be \emph{transversally coercive} at $(u,f)$ if $D_{u,f}(u)=0$ and there exists $c_{u,f}>0$ such that $D_{u,f}(x)\geq c_{u,f}\|Q_{u,f}x\|$ whenever $f(x)\geq0$. The number $c_{u,f}$ is called a transversal coercivity constant.
\end{definition}

Equivalently, $\|Q_{u,f}x\| \leq \frac{1}{c_{u,f}}D_{u,f}(x)$, whenever $f(x)\geq0$, so that every upper estimate on $D_{u,f}(x)$ yields a corresponding
upper estimate on the transversal component. The canonical choice $D_{u,f}(x):=\|Q_{u,f}x\|$ has coercivity constant $1$ for every norming pair in every normed
space.
More generally, if
$p:\ker f\to[0,+\infty)$ is continuous and positively homogeneous and
satisfies $p(z)\geq c\|z\|$, $z\in\ker f$, then $D_{u,f}:=p\circ Q_{u,f}$ is transversally coercive. This includes,
for instance, equivalent transversal norms and deviations obtained
from bounded below linear operators.

\subsection{Axial representation of localized convex cones}
\label{subsec:axial-representation}

Transversal coercivity can first be read as a representation property:
every convex cone already localized around a direction admits such a
description.

\begin{proposition}
\label{prop:localized-convex-axial-representation}
Let $C\subset X$ be a closed convex cone, let
$u\in S_X\cap\operatorname{int}C$, and assume that $C\subset\operatorname{cone}B(u,\varepsilon)$ for some $0<\varepsilon<1$. For every $f\in S_{X^*}$ with $f(u)=1$, there exists a continuous,
positively homogeneous and subadditive mapping
$D_{C,u,f}:X\to[0,+\infty)$ which is transversally coercive and satisfies $$C=\{x\in X:D_{C,u,f}(x)\leq f(x)\}.$$ More precisely, $D_{C,u,f}$ may be chosen with transversal coercivity
constant $c_{u,f}=\frac{1-\varepsilon}{2\varepsilon}$. Consequently,
$\varphi_{C,u,f}:=f-D_{C,u,f}$ is a continuous superlinear adapted
defining function for $C$ at $u$.
\end{proposition}
\begin{proof}
Set $Q:=Q_{u,f}$ and
$V:=\{z\in\ker f:u+z\in C\}$. Then $V$ is a closed convex
neighborhood of $0_X$ in $\ker f$. Since $u\in\operatorname{int}C$,
there exists $\rho>0$ such that
$\rho(B_X\cap\ker f)\subset V$.

Let $z\in V$. Since
$u+z\in C\subset\operatorname{cone}B(u,\varepsilon)$, there exist
$t>0$ and $y\in B(u,\varepsilon)$ such that $u+z=ty$. As
$1=f(u+z)=tf(y)$ and
$f(y)\geq1-\|y-u\|\geq1-\varepsilon$, we have
$t\leq(1-\varepsilon)^{-1}$. Applying $Q$ and using
$Qu=0_X$, $Qz=z$, and $\|Q\|\leq2$, we obtain $\|z\|
=\|tQ(y-u)\|
\leq\frac{2\varepsilon}{1-\varepsilon}$. Hence $\rho(B_X\cap\ker f)
\subset V
\subset
\frac{2\varepsilon}{1-\varepsilon}(B_X\cap\ker f)$.

Let $p_V(z):=\inf\{\lambda>0:z\in\lambda V\}$, $z\in\ker f$,
be the Minkowski functional of $V$ and define
$D_{C,u,f}:=p_V\circ Q$. Since $V$ is convex and contains a
neighborhood of $0_X$ in $\ker f$, $p_V$ is finite-valued,
continuous, positively homogeneous and subadditive. The preceding
upper estimate for $V$ gives $D_{C,u,f}(x)
\geq
\frac{1-\varepsilon}{2\varepsilon}\|Qx\|,
\qquad x\in X$.

We next prove the representation. If
$x\in C\setminus\{0_X\}$, localization yields $f(x)>0$. Since $C$ is
a cone,
$x/f(x)=u+Qx/f(x)\in C$, and therefore
$Qx/f(x)\in V$. Hence
$D_{C,u,f}(x)\leq f(x)$, and the same inequality is trivial at
$0_X$.

Conversely, suppose that $D_{C,u,f}(x)\leq f(x)$. Since
$D_{C,u,f}\geq0$, one has $f(x)\geq0$. If $f(x)=0$, transversal
coercivity gives $Qx=0_X$, and hence
$x=f(x)u+Qx=0_X\in C$. If $f(x)>0$, then
$p_V(Qx/f(x))\leq1$. Since $V$ is closed, convex, and contains a neighborhood of $0_X$ in
$\ker f$, its Minkowski functional satisfies
$V=\{z\in\ker f:p_V(z)\leq1\}$. Thus $Qx/f(x)\in V$, so
$u+Qx/f(x)\in C$, and, since $C$ is a cone, $x\in C$. Therefore $C=\{x\in X:D_{C,u,f}(x)\leq f(x)\}$.

Finally, $\varphi_{C,u,f}:=f-D_{C,u,f}$ is continuous and positively
homogeneous, and it is superlinear because $D_{C,u,f}$ is subadditive.
Moreover, $\varphi_{C,u,f}(u)=1$ and
$\varphi_{C,u,f}(x)\leq f(x)\leq\|x\|$ for every $x\in X$.
Hence $\varphi_{C,u,f}$ is an adapted defining function for $C$ at $u$.
\end{proof}
\begin{remark}
\label{rem:BP-axial-interpretation}
Combining Proposition~\ref{prop:localized-convex-axial-representation}
with Theorem~\ref{thm:BP-directionally-fine} shows that, whenever the
Bishop--Phelps family is directionally fine, every localized
Bishop--Phelps cone selected there admits an exact axial representation
of the preceding form. The resulting defining function $f-D_{C,u,f}$
need not coincide with the usual Bishop--Phelps function
$g-\alpha\|\cdot\|$; it is an alternative representation obtained from
the transversal section $C\cap\{f=1\}$.
\end{remark}

\subsection{Construction by transversal coercivity}
\label{subsec:normalization-principles}

We now use the same idea in the opposite direction. A transversally
coercive deviation produces localized cones through either axial or
norm normalization.

\begin{theorem}
\label{thm:two-normalization-principles}
Let $(u,f)$ be a norming pair and let
$D_{u,f}:X\to[0,+\infty)$ be transversally coercive with constant
$c_{u,f}>0$.

\begin{enumerate}
\item[(i)]
For $\delta>0$, set
\[
\varphi^a_{D,u,f,\delta}
:=f-\delta^{-1}D_{u,f},
\qquad
C^a_D(u,f,\delta)
:=\{\varphi^a_{D,u,f,\delta}\geq0\}.
\]
Then $C^a_D(u,f,\delta)$ is a closed cone,
$u\in\operatorname{int}C^a_D(u,f,\delta)$, and
\[
0<\delta<c_{u,f}\varepsilon
\quad\Longrightarrow\quad
C^a_D(u,f,\delta)\subset\operatorname{cone}B(u,\varepsilon),
\qquad 0<\varepsilon<1.
\]
For such $\delta$ the cone is pointed and
$\varphi^a_{D,u,f,\delta}$ is an adapted defining function. If
$D_{u,f}$ is subadditive, then $\varphi^a_{D,u,f,\delta}$ is
superlinear and $C^a_D(u,f,\delta)$ is convex.

\item[(ii)]
For $r>0$, set
\[
\varphi^n_{D,u,f,r}(x)
:=
\min\{f(x),\,\|x\|-r^{-1}D_{u,f}(x)\},
\]
and
\[
C^{\rm n}_D(u,f,r)
:=
\{\varphi^n_{D,u,f,r}\geq0\}.
\]
Then $C^{\rm n}_D(u,f,r)$ is a closed cone,
$u\in\operatorname{int}C^{\rm n}_D(u,f,r)$, and
\[
0<r<\frac{c_{u,f}\varepsilon}{1+\varepsilon}
\quad\Longrightarrow\quad
C^{\rm n}_D(u,f,r)\subset\operatorname{cone}B(u,\varepsilon).
\]
For such $r$ the cone is pointed and
$\varphi^n_{D,u,f,r}$ is an adapted defining function.
\end{enumerate}
\end{theorem}

\begin{proof}
For (i), if $0_X\neq x\in C^a_D(u,f,\delta)$, then $c_{u,f}\|Q_{u,f}x\|
\leq D_{u,f}(x)\leq\delta f(x)$, so $f(x)>0$ and $\left\|\frac{x}{f(x)}-u\right\|
=\frac{\|Q_{u,f}x\|}{f(x)}
\leq\frac{\delta}{c_{u,f}}$. This proves the localization. Closedness and interiority follow from
continuity and $\varphi^a_{D,u,f,\delta}(u)=1$; pointedness follows
from the inclusion in $\operatorname{cone}B(u,\varepsilon)$, while
subadditivity of $D_{u,f}$ gives superlinearity.

For (ii), let $0_X\neq x\in C^{\rm n}_D(u,f,r)$ with $r<c_{u,f}$. Then
$f(x)>0$. Setting $t:=\|Q_{u,f}x\|/f(x)$ and using
$\|x\|\leq f(x)(1+t)$ gives $c_{u,f}t\leq r(1+t)$, $t\leq\frac{r}{c_{u,f}-r}$. The stated bound on $r$ makes the last quantity smaller than
$\varepsilon$, and hence $x\in\operatorname{cone}B(u,\varepsilon)$.
The remaining assertions follow directly from the definitions.
\end{proof}

The canonical choice $D_{u,f}(x)=\|Q_{u,f}x\|$ yields two explicit
universal models. The uniform axial deviation cone is
\[
C^{\rm ax}(u,f,\delta)
:=\{x\in X:f(x)-\delta^{-1}\|Q_{u,f}x\|\geq0\},
\]
and the norm-normalized axial cone is
\[
C^{\rm n}(u,f,r)
:=\{x\in X:f(x)\geq0,\ \|Q_{u,f}x\|\leq r\|x\|\}.
\]

\begin{corollary}\label{cor:canonical-axial-families}
For every normed space $X$, set
\[
\begin{aligned}
\mathcal F_{\rm ax}
&:=\{C^{\rm ax}(u,f,\delta):(u,f)\text{ norming},\ \delta>0\},\\
\mathcal F_{\rm n}
&:=\{C^{\rm n}(u,f,r):(u,f)\text{ norming},\ 0<r<1\}.
\end{aligned}
\]
Both families are directionally fine, and $\mathcal F_{\rm ax}$ is convex directionally
fine. More precisely, for $0<\varepsilon<1$,
$0<\delta<\varepsilon$ we have
$C^{\rm ax}(u,f,\delta)\subset\operatorname{cone}B(u,\varepsilon)$,
whereas $0<r<\varepsilon/(1+\varepsilon)$ implies
$C^{\rm n}(u,f,r)\subset\operatorname{cone}B(u,\varepsilon)$. Moreover, $C^{\rm ax}(u,f,\delta)
 =\operatorname{cone}\bigl((u+\ker f)\cap B(u,\delta)\bigr)$. Similarly, $C^{\rm n}(u,f,r)
 =\operatorname{cone}\bigl((u+\ker f)
 \cap\{z:\|z-u\|\le r\|z\|\}\bigr)$.

If $X=H$ is a Hilbert space, then
$C^{\rm n}(u,f,r)=C^{\rm ax}(u,f,r/\sqrt{1-r^2})$; hence the two constructions
generate the same family up to reparametrization. In general normed spaces they need
not coincide.
\end{corollary}

\begin{proof}
Apply Theorem~\ref{thm:two-normalization-principles} with
$D_{u,f}(x)=\|Q_{u,f}(x)\|$ and $c_{u,f}=1$.
The last identities follow by normalizing every nonzero element by
$f(x)$.
In the Hilbert case, the last identity follows from orthogonality. Indeed,
for a norming pair $(u,f)$ one has \(f(x)=\langle u,x\rangle\) and
\(Q_{u,f}x\perp u\), hence
\(
\|x\|^2
=
f(x)^2+\|Q_{u,f}x\|^2.
\)
Therefore,
\(
\|Q_{u,f}x\|\le r\|x\|
\iff
(1-r^2)\|Q_{u,f}x\|^2\le r^2 f(x)^2.
\)
Since points of \(C^{\rm n}(u,f,r)\) satisfy \(f(x)\ge0\), this is
equivalent to
\(
\|Q_{u,f}x\|
\le
\frac{r}{\sqrt{1-r^2}}\,f(x),
\)
and consequently
\(
C^{\rm n}(u,f,r)
=
C^{\rm ax}(
u,f,\frac{r}{\sqrt{1-r^2}}
).
\)
\end{proof}

\begin{remark}
\label{rem:scope-transversal-construction}
The construction is considerably more flexible than the two canonical
models. For example, if
$E_{u,f}:X\to[0,+\infty)$ is continuous and positively homogeneous
with $E_{u,f}(u)=0$, then
\[
D_{u,f}(x):=\|Q_{u,f}x\|+E_{u,f}(x),
\]
is transversally coercive with constant $1$ and therefore yields, by
Theorem~\ref{thm:two-normalization-principles}, both axially normalized
and norm-normalized directionally fine families. Likewise,
$D_{u,f}=p\circ Q_{u,f}$ is admissible whenever $p$ is a continuous
positively homogeneous functional on $\ker f$ satisfying
$p(z)\geq c\|z\|$ for some $c>0$. Thus transversal coercivity provides
a general construction principle rather than a single cone family.

All these families may be used directly in the approximation and
finite-separation results of Section~\ref{sec:separation-approximation};
when a finite collection $C_i=\{\varphi_i\geq0\}$ is selected, its
union is represented explicitly by $\bigcup_{i=1}^m C_i
=\left\{\max_{1\leq i\leq m}\varphi_i\geq0\right\}$.
\end{remark}

\subsection{Penalization, Recovery, and Convergence}
\label{subsec:UAD-optimization}
We conclude this section by illustrating how the explicit geometry of the
uniform axial deviation family can be used in optimization. Combining the
finite outer approximation of a locally compact cone with the associated
max-type defining function yields an error bound for the approximating
cone. This makes it possible to replace a conically constrained problem by
an unconstrained penalized problem and to recover from each bounded
penalized minimizer a feasible point of the original problem with explicit
distance and optimality estimates.

The following result makes this mechanism quantitative and also provides
an explicit estimate for the approximation of the optimal value.
\begin{theorem}
\label{thm:UAD_approximate_optimization}
Let $L\subset X$ be a nontrivial locally compact cone and let
$g:X\to\mathbb R$ be $L_g$-Lipschitz continuous, with $L_g>0$.
Consider $(P)\ \inf\{g(x):x\in L\}$. For every $\eta\in (0,1)$, there exist a finite uniform
axial deviation outer approximation $D_L^\eta$ of $L$ and a continuous
positively homogeneous function $\Phi_\eta:X\to\mathbb R$ such that
$L\subset D_L^\eta=\{\Phi_\eta\geq0\}$. Moreover, for every $M>L_g$, the penalized problem
$(P_\eta)\ \inf_{x\in X}\{g(x)+M[-\Phi_\eta(x)]_+\}$ satisfies $\inf(P_\eta)=\inf_{x\in D_L^\eta}g(x)$, and every solution
of $(P_\eta)$ belongs to $D_L^\eta$.

If $(P_\eta)$ has a solution $x_\eta$ with $\|x_\eta\|\leq R$ for some
$R>0$, set $v:=\inf_{x\in L}g(x)$ and $v_\eta:=\min(P_\eta)$. Then $v\in\mathbb R$ and there exists $l_\eta\in L$ such that $\|x_\eta-l_\eta\|\leq\eta\|x_\eta\|\leq R\eta$. Moreover,
\[
0\leq g(l_\eta)-v\leq L_g\eta\|x_\eta\|\leq L_gR\eta,\qquad
0\leq v-v_\eta\leq L_g\eta\|x_\eta\|\leq L_gR\eta.
\]
Consequently, $l_\eta$ is an $\varepsilon$-solution of $(P)$, that is,
$g(l_\eta)\leq v+\varepsilon$, whenever
$0<\eta\leq\varepsilon/(L_gR)$.
\end{theorem}
\begin{proof}
Fix $\eta\in (0,1)$ and let
$D_L^\eta=\bigcup_{i=1}^m C_i$ be a finite uniform axial deviation
outer approximation of $L$, where $C_i=\{\varphi_i\geq0\}$,
$\varphi_i(x)=f_i(x)-\delta_i^{-1}\|Q_{u_i,f_i}x\|$,
$u_i\in B_L$, and $0<\delta_i<\eta$. Set
$\Phi_\eta:=\max_{1\leq i\leq m}\varphi_i$. Then $\Phi_\eta$ is
continuous and positively homogeneous, and
$L\subset D_L^\eta=\{\Phi_\eta\geq0\}$.

We first show that
$d(x,D_L^\eta)\leq[-\Phi_\eta(x)]_+$ for every $x\in X$. This is
immediate if $\Phi_\eta(x)\geq0$. Otherwise, choose $i$ such that
$\Phi_\eta(x)=\varphi_i(x)$. Since
$Q_{u_i,f_i}(x+tu_i)=Q_{u_i,f_i}x$ and $f_i(u_i)=1$, we have
$\varphi_i(x+tu_i)=\varphi_i(x)+t$. Taking
$t=-\Phi_\eta(x)>0$ gives $x+tu_i\in C_i$, and therefore
$d(x,D_L^\eta)\leq t=[-\Phi_\eta(x)]_+$.

Put $\widehat v_\eta:=\inf_{D_L^\eta}g$. If
$\widehat v_\eta=-\infty$, then $\inf(P_\eta)=-\infty$, since the
penalty vanishes on $D_L^\eta$. Otherwise, Lipschitz continuity and the
preceding error bound give
$g(x)\geq\widehat v_\eta-L_g[-\Phi_\eta(x)]_+$. Hence, for $M>L_g$,
$g(x)+M[-\Phi_\eta(x)]_+
\geq\widehat v_\eta+(M-L_g)[-\Phi_\eta(x)]_+
\geq\widehat v_\eta$.
Since the penalty vanishes on $D_L^\eta$, in either case
$\inf(P_\eta)=\widehat v_\eta$.

If $x_\eta$ is a solution, then necessarily
$\widehat v_\eta\in\mathbb R$. Moreover,
$\Phi_\eta(x_\eta)<0$ would make the preceding inequality strict, a
contradiction. Thus $x_\eta\in D_L^\eta$ and
$v_\eta=g(x_\eta)=\widehat v_\eta$. Since
$L\subset D_L^\eta$ and $0_X\in L$, we have
$-\infty<v_\eta\leq v\leq g(0_X)<+\infty$, so $v\in\mathbb R$.

Choose $i$ such that $x_\eta\in C_i$ and write
$x_\eta=f_i(x_\eta)u_i+Q_{u_i,f_i}x_\eta$. Set
$l_\eta:=f_i(x_\eta)u_i$. Since $x_\eta\in C_i$,
$f_i(x_\eta)\geq\delta_i^{-1}\|Q_{u_i,f_i}x_\eta\|\geq0$, and hence
$l_\eta\in L$. Moreover,
$\|x_\eta-l_\eta\|=\|Q_{u_i,f_i}x_\eta\|
\leq\delta_i f_i(x_\eta)\leq\eta\|x_\eta\|\leq R\eta$.

Finally, $l_\eta\in L$ and $g(x_\eta)=v_\eta\leq v$. Therefore,
Lipschitz continuity gives
$0\leq g(l_\eta)-v
\leq g(l_\eta)-g(x_\eta)
\leq L_g\eta\|x_\eta\|
\leq L_gR\eta$.
The same estimate yields
$0\leq v-v_\eta\leq L_g\eta\|x_\eta\|\leq L_gR\eta$.
The last assertion follows immediately.
\end{proof}
The preceding estimates also yield a sequential stability property as the
directional approximation parameter tends to zero. In particular, uniform
boundedness of penalized minimizers, together with local compactness of
$L$, is enough to obtain convergence of optimal values and subsequential
convergence to solutions of the original problem.
\begin{corollary}
\label{cor:UAD_penalized_convergence}
Under the assumptions of
Theorem~\ref{thm:UAD_approximate_optimization}, let $\eta_n\downarrow0$
and, for each $n$, choose $\Phi_{\eta_n}$ as in that theorem. Fix
$M>L_g$ and suppose that $(P_{\eta_n})$ has a solution $x_n$ with
$\sup_n\|x_n\|\leq R$ for some $R>0$. Set
$v_n:=\min(P_{\eta_n})$. Then there exist $l_n\in L$ such that
\[
\|x_n-l_n\|\to0,\qquad g(l_n)\to v,\qquad v_n\to v.
\]
Moreover, $(x_n)$ and $(l_n)$ are relatively compact, $(P)$ has a
solution, and every cluster point of either sequence belongs to
$\operatorname{argmin}_{x\in L}g(x)$. In particular, if $(P)$ has a
unique solution $\bar l$, then $x_n\to\bar l$ and $l_n\to\bar l$.
\end{corollary}

\begin{proof}
By Theorem~\ref{thm:UAD_approximate_optimization}, there exist $l_n\in L$
such that $\|x_n-l_n\|\leq R\eta_n$, $0\leq g(l_n)-v\leq L_gR\eta_n$, $0\leq v-v_n\leq L_gR\eta_n$. Hence $\|x_n-l_n\|\to0$, $g(l_n)\to v$, and $v_n\to v$. Moreover, $\|l_n\|\leq R+R\eta_n$, so eventually $l_n\in L\cap2RB_X$, which is
compact by Proposition~\ref{prop:local_compactness}. Thus $(l_n)$ is
relatively compact and, since $\|x_n-l_n\|\to0$, so is $(x_n)$.

If $l_{n_k}\to\bar l$, then $\bar l\in L$ and, by continuity of $g$,
$g(\bar l)=v$; hence $\bar l\in\operatorname{argmin}_{x\in L}g(x)$, so
$(P)$ has a solution. If instead $x_{n_k}\to\bar x$, then
$l_{n_k}\to\bar x$, and the same argument gives
$\bar x\in\operatorname{argmin}_{x\in L}g(x)$. Thus every cluster point
of either sequence solves $(P)$. If the solution is unique, relative
compactness yields convergence of both sequences to $\bar l$.
\end{proof}

\section{Conclusions}

We have introduced directional fineness as a local property of cone families
that allows directional models to be assembled, through compactness of the
relevant set of directions, into finite approximation and separation
constructions. Beyond this local-to-finite mechanism, two complementary
features of the framework are particularly relevant. First, the
Bishop--Phelps family is directionally fine precisely under property~(G),
thereby linking the construction with the geometry of the norm. Second,
transversal coercivity provides an axial construction principle yielding
directionally fine families in arbitrary normed spaces, independently of such
geometric assumptions. The uniform axial deviation model also leads to quantitative penalization, recovery, and convergence results in conically
constrained optimization. Overall, the relevant geometry is essentially
directional: it is governed by normalized directions on the unit sphere,
rather than by the radial size of the sets involved.

Several questions remain open. Although the Kadec setting already provides a weak-compactness alternative, it would be interesting to determine whether analogous weakenings are possible more generally, through weak or cone-related compactness notions, and whether uniform versions of directional fineness allow further refinements. The role of Kadec-type renormings also deserves investigation. Finally, the explicit positively homogeneous separators suggest further applications to nonlinear scalarization, particularly to proper and approximate efficiency.

\section*{Acknowledgements}
The authors acknowledge financial support from the following projects:
PID2021-122126NB-C32, funded by MICIU/AEI/\nolinkurl{10.13039/501100011033} and by
FEDER, ``A way of making Europe''; and PID2025-168246NB-I00, funded by
MICIU/AEI/\nolinkurl{10.13039/501100011033} and by ERDF/EU.

\bibliographystyle{plain}
\bibliography{references}

\end{document}